\documentclass[12pt,a4paper]{amsart}
\usepackage[T1]{fontenc}    
\usepackage[utf8]{inputenc}
\usepackage{enumerate}
\usepackage{amssymb}
\usepackage{fullpage}
\usepackage{color}

\newcommand{\set}[1]{\left\lbrace #1\right\rbrace}
\providecommand{\abs}[1]{\left\lvert#1\right\rvert}

\newcommand{\remove}[1]{ }
\newcommand{\qtq}[1]{\quad \text{#1}\quad }

\newtheorem{theorem}{Theorem}[section]
\newtheorem{proposition}[theorem]{Proposition}
\newtheorem{lemma}[theorem]{Lemma}
\newtheorem{corollary}[theorem]{Corollary}

\theoremstyle{remark}

\newtheorem{remarks}[theorem]{Remarks}

\newtheorem{examples}[theorem]{Examples}
\newtheorem{definition}[theorem]{Definition}

\def\uu{\mathcal U}
\def\ee{\mathcal E}
\def\vv{\mathcal V}

\def\nn{\mathcal N}
\def\uuu{\overline{\mathcal U}}
\def\uuuq{\overline{{\mathcal U}_q}}
\def\bfu{\mathbf U}
\def\bfuu{\overline{\mathbf U}}
\def\bfv{\mathbf V}
\def\bfj{\mathbf J}

\def\NN{\mathbb N}
\def\ZZ{\mathbb Z}
\def\RR{\mathbb R}
\numberwithin{equation}{section}

\begin{document}

\title{Topology of univoque sets in real base expansions}
\author[M. de Vries]{Martijn de Vries}
\address{Tussen de Grachten 213, 1381DZ Weesp, the Netherlands} 
\email{martijndevries0@gmail.com}

\author{Vilmos Komornik}
\address{College of Mathematics and Statistics,
Shenzhen University,
Shenzhen 518060,
People’s Republic of China,
and
Département de mathématique,
Université de Strasbourg,
7 rue René Descartes,
67084 Strasbourg Cedex, France}
\email{komornik@math.unistra.fr}

\author[P. Loreti]{Paola Loreti}
\address{Sapienza Università di Roma,
Dipartimento di Scienze di Base
e Applicate per l'Ingegne\-ria,
via A. Scarpa n. 16,
00161 Roma, Italy}
\email{paola.loreti@sbai.uniroma1.it}

\thanks{The work of the second author was partially supported by the National Natural Science Foundation
of China (NSFC)  \#11871348, and by the grants CAPES: No. 88881.520205/2020-01 and MATH AMSUD: 21-MATH-03.}
\date{Version of 2022-03-15}

\begin{abstract}
Given a positive integer $M$ and a real number $q  \in (1,M+1]$, an expansion of a real number $x \in \left[0,M/(q-1)\right]$ over the alphabet $A=\set{0,1,\ldots,M}$ is a sequence $(c_i) \in A^{\NN}$ such that $x=\sum_{i=1}^{\infty}c_iq^{-i}$. Generalizing many earlier results, we investigate in this paper the topological properties of the set $\uu_q$ consisting of numbers $x$  having a unique expansion of this form, and the combinatorial properties of the set $\uu_q'$ consisting of their corresponding expansions. We also provide shorter proofs of the main results of Baker in \cite{[B]} by adapting the method given in \cite{[EJK]} for the case $M=1$. 
\end{abstract}
\maketitle

\section{Introduction and statement of the main results}\label{s1}

Starting with a seminal paper of R\'enyi \cite{[R]} many papers have been devoted to representations of real numbers $x$ of the form
\begin{equation*}
x=\sum_{i=1}^{\infty} \frac{c_i}{q^i},
\end{equation*}
where the base $q>1$ is a given real number, and $(c_i)$ is a sequence of integers with $0 \le c_i \le M$ $(i \in \NN:=\ZZ_{\ge 1})$, where $M$ is a given positive integer. The sequence $(c_i)$ is often called an expansion of $x$. Such representations of real numbers have many intimate connections to combinatorics, number theory, probability and ergodic theory, topology and symbolic dynamics. See, e.g., the review papers \cite{[S], K2011, DK2016}.

In the 1990's, Erd\H os, Horv\'ath and Jo\'o \cite{[EHJ]} found bases $q \in (1,2)$ such that $x=1$ has exactly one representation of the above form with digits $c_i$ belonging to $\set{0,1}$. Following this discovery,  the combinatorial and topological structure of (the set of) all numbers $x$ having exactly one representation of the form $\sum_{i=1}^{\infty} c_i\cdot q^{-i}$ with digits $c_i \in \set{0,1,\ldots, M}$ -- and the corresponding set of sequences $(c_i)$ --  was eventually clarified in \cite{[KL1], [KL3], [DVK1], [DVK2]} under the additional assumption that $M <  q\le M+1$.
Later developments led to the necessity to relax this assumption. 
This was done in \cite{[DKL]} for the expansions of $x=1$.
Building on the results of \cite{[DKL]} here we clarify the structure of the set of  real numbers $x$ with a unique expansion for any choice of $M\ge 1$ and $q>1$. Although the general research strategy is the same as in \cite{[DVK1]}, some new arguments are needed and several new properties are uncovered.   

In order to state our results we introduce some notation and terminology.
In this paper we fix a positive integer $M$, and we consider the corresponding alphabet $A:=\set{0,1,\ldots,M}$. The elements of $A$ are often called \emph{digits}. Since $M$ is fixed, we usually do not indicate the dependence on $M$ of the notions we are going to introduce.

A \emph{sequence} always means an element of the set $A^{\NN}$; it will often be written in the form $(c_i)$ or $c_1c_2\cdots$. 
By a \emph{block} or \emph{word}  we mean an element of $\cup_{k \in \NN} A^k$. A block or word has \emph{length} $n$ if it belongs to $A^n$. 
We will also use the \emph{conjugate} or \emph{reflection} of any digit $c$, word $c_1\cdots c_n$ or sequence $(c_i)$, defined by
\begin{equation*}
\overline{c}:=M-c,\quad
\overline{c_1 \cdots c_n}:=\overline{c_1} \cdots \overline{c_n}\qtq{and}
\overline{c_1 c_2 \cdots}:=\overline{c_1} \,
\overline{c_2} \cdots.
\end{equation*} 
Finally, if $c\in A$, then we set $c^+:=c+1$ if $c<M$, and $c^-:=c-1$ if $c>0$, so that $c^+\in A$ and $c^-\in A$.
More generally, if $w=c_1\cdots c_n$ is a word of length $n\ge 2$, then we write 
\begin{equation*}
w^+=(c_1\cdots c_n)^+:=c_1\cdots c_{n-1}c_n^+ 
\end{equation*}
if $c_n<M$, and 
\begin{equation*}
w^-=(c_1\cdots c_n)^-:=c_1\cdots c_{n-1}c_n^-
\end{equation*}
if $c_n>0$.

We will use systematically the lexicographical order between sequences: we write $(a_i) < (b_i)$ or $(b_i) > (a_i)$ if there exists an index $n \in \NN$ such that $a_i=b_i$ for $i < n$, and $a_n < b_n$. 
We also equip for each $n \in \NN$ the set $A^{n}$ of blocks of length $n$ with the lexicographical order.
We apply the usual notation from symbolic dynamics. For example, $0^{\infty}$ would indicate the sequence $c$ with $c_i=0$ for all $i\ge 1$, $(10)^{\infty}$ would indicate the sequence  $c$ with $c_{2i-1}=1$ and $c_{2i}=0$ for all $i\ge 1$, and so on. A sequence $(c_i)$ is called
\begin{itemize}
\item \emph{finite} if it has a last nonzero element, and \emph{infinite} otherwise;
\item \emph{co-finite} if its conjugate is finite, and \emph{co-infinite} otherwise;
\item \emph{doubly infinite} if it is both infinite and co-infinite.
\end{itemize}
This unusual terminology enables us to simplify the statements of many results.
Note that $0^{\infty}$ does not have a last nonzero element and is thus infinite, hence doubly infinite. Similarly, $M^{\infty}$ is a doubly infinite sequence.
The other doubly infinite sequences are those sequences that have both infinitely many digits $c_i>0$  and infinitely many  digits $c_i<M$.

Given a real number $q >1$, an \emph{expansion in base} $q$ \emph{over the alphabet} $A$ (or simply \emph{expansion in base $q$} or \emph{expansion} if there is no risk of confusion) of a real number $x$ is a sequence $c=(c_i)$ satisfying the equality
\begin{equation*}
\pi_q(c):=\sum_{i=1}^{\infty} \frac{c_i}{q^i}=x.
\end{equation*}
We sometimes write $\pi_q(c_1c_2\cdots)$ or $\pi_q((c_i))$ in place of $\pi_q(c)$. 

If $q>M+1$, then there exist numbers $x$ satisfying the inequalities 
\begin{equation}\label{10}
\frac{M}{q^2}+\frac{M}{q^3}+\cdots<x<\frac{1}{q},
\end{equation}
and they have no expansions: for any sequence $(c_i)$ we have $\pi_q(c_1c_2\cdots)<x$ if $c_1=0$, and $\pi_q(c_1c_2\cdots)>x$ if $c_1>0$. The inequalities \eqref{10} also imply that each $c \in A^{\NN}$ is the unique expansion of $\pi_q(c)$. The topological structure of the set $\pi_q \left (A^\NN \right)$ consisting of numbers with an expansion (which is always unique as we just observed) is in this case rather straightforward and resembles that of the classical triadic Cantor set $C:=\set{\sum_{i=1}^{\infty} a_i \cdot 3^{-i}: a_i\in\set{0,2}, i \ge 1}$. The finite sequences in $A^{\NN}$ should be compared with the right endpoints of the connected components of $[0,1]\setminus C$ and the co-finite sequences with its left endpoints. 
For this reason, we restrict ourselves in the sequel of this paper to bases $q \in (1,M+1]$ in which case $J_q:=\pi_q \left(A^\NN \right)=[0,M/(q-1)]$; see \cite{[R],[P],[EJK], [BK]}. Moreover,  every $x \in J_q$ has a lexicographically largest expansion $b(x,q)$ and a lexicographically largest infinite expansion $a(x,q)$, called the \emph{greedy} and \emph{quasi-greedy} expansions of $x$ in base $q$, respectively.
For example, in the  case of the classical binary expansions (so $q=2$ and $M=1$), the fractions $x=\frac{k}{2^m}\in(0,1)$ with positive integers $m$ and $k$ have exactly two expansions: a finite and an infinite one; they are the greedy and quasi-greedy expansions of $x$, respectively.
All other numbers $x\in J_2=[0,1]$ have a unique expansion.

Of course, whether an expansion is greedy or quasi-greedy
depends on $q$ and $M$.
However, when $q$ and $M$ are understood from the context, we simply speak of (quasi-) greedy expansions.

Let us give some examples, by describing  the expansions of $1$ in various bases over the alphabet $A=\set{0,1}$ (so $M$=1); see \cite{[EJK], [DKL]} for details.

\begin{examples}\label{e10}\mbox{}
\begin{enumerate}[\upshape (i)]
\item There exists a base $1<q<2$ such that $\pi_q(1(10)^{\infty})=1$.
In this base $1(10)^{\infty}$ is the unique expansion of $1$, and it is doubly infinite.

\item In the \emph{Tribonacci base} $q \approx 1.839$, defined by the equation $q^3=q^2+q+1$, $1$ has $\aleph_0$ expansions: $(110)^{\infty}$, and the sequences
\begin{equation*}
(110)^k(111)0^{\infty},\quad k=0,1,\ldots .
\end{equation*} 
Here $(110)^{\infty}$ is a doubly infinite expansion, and all other expansions are finite.

\item In the \emph{Golden ratio base} $q=(1+\sqrt{5})/2$, defined by the equation $q^2=q+1$, $1$ has $\aleph_0$ expansions again: $(10)^{\infty}$, and the sequences
\begin{equation*}
(10)^k(11)0^{\infty}\qtq{and}
(10)^k01^{\infty},\quad k=0,1,\ldots .
\end{equation*} 
Here $(10)^{\infty}$ is a doubly infinite expansion.
There are many infinite expansions, but $(10)^{\infty}$ is the only doubly infinite expansion.

\item In every \emph{small base} $1<q<(1+\sqrt{5})/2$, $1$ has $2^{\aleph_0}$ expansions; hence it has $2^{\aleph_0}$ doubly infinite expansions as well. 
\end{enumerate}
\end{examples}
The choice of the alphabet $A$ in the examples above and in general is pertinent. For instance, if $M=2$, then 1 in the Golden ratio base has $2^{\aleph_0}$ expansions including all expansions $(c_i)$ satisfying $c_{4k+1} \cdots c_{4k+4} \in \set{1010, 0120}$ for all $k \ge 0$.  

In \cite{[DKL]} we investigated the set of \emph{univoque bases}
\begin{align*}
&\uu:=\set{q>1\ :\ 1\text{ has a unique expansion in base }q}
\intertext{and the larger set}
&\vv:=\set{q>1\ :\ 1\text{ has a unique doubly infinite expansion in base }q}.
\intertext{It was shown that $\vv$ is closed and that the closure $\uuu$ of  $\uu$ is a \emph{Cantor set}, i.e., a nonempty closed set having neither interior nor isolated points. Moreover, $\uuu$ can be characterized as follows:} 
&\uuu=\set{q>1\ :\ 1\text{ has a unique infinite expansion in base }q}.
\end{align*}

In the following table we illustrate these notions by showing the number of doubly infinite expansions, infinite expansions and all expansions of $1$ in the above four examples:

\begin{center}
\begin{tabular}{|c|c|c|c|c|}
\hline
Examples \ref{e10}& d.i. expansions & i. expansions &all expansions& $q$ belongs to \\
\hline
(i)&1&1&1&$\uu$\\
(ii)&1&1&$\infty$&$\uuu\setminus\uu$\\
(iii)&1&$\infty$&$\infty$&$\vv\setminus\uuu$\\
(iv)&$\infty$&$\infty$&$\infty$&$\left(1,(1+\sqrt{5})/2\right)$\\
\hline
\end{tabular}
\end{center}

The purpose of this paper is to carry out a similar study of the \emph{univoque set}
\begin{align*}
&\uu_q:=\set{x\in J_q\ :\ x\text{ has a unique expansion in base }q}
\intertext{for each fixed base $q>1$.
We will prove for example that $\uu_q$ is a closed set if and only if $q\notin\uuu$.
In order to state the main topological properties of the sets $\uu_q$ we introduce the related sets $\vv_q$ as follows: for $q \in (1,M+1)$, we set}
&\vv_q:=\set{x\in J_q\ :\ x\text{ has a unique doubly infinite expansion in base }q},
\intertext{and for $q=M+1$ we set $\vv_q:=J_q=[0,1]$. If $q=M+1$, numbers $x \in J_q$ with a finite expansion have no doubly infinite expansion, while for $1 <q<M+1$, the quasi-greedy expansion $a(x,q)$ is always doubly infinite; see Proposition \ref{p22} (ii). Hence,}
&\vv_q=\set{x\in J_q\ :\ x\text{ has at most one doubly infinite expansion in base }q},
\intertext{for \emph{each} $q \in (1,M+1]$.} 
\end{align*}

The most important relations between the sets $\uu_q$ and $\vv_q$ are described in the following Theorems \ref{t11}, \ref{t13} and \ref{t14}:

\begin{theorem} \label{t11}\mbox{}
\begin{enumerate}[\upshape (i)]
\item If $q \in \uuu$, then $\overline{\uu_q} = \vv_q$.
\item If $q \in \uuu$, then $\abs{\vv_q\setminus\uu_q}=\aleph_0$
and $\vv_q\setminus \uu_q$ is dense in $\vv_q$.
\item If $q \in \uu$, then each element $x \in \vv_q
\setminus \uu_q$ has exactly $2$ expansions. 
\item If $q \in \uuu \setminus \uu$, then each
element $x \in \vv_q \setminus \uu_q$ has exactly $\aleph_0$ expansions.
\end{enumerate}
\end{theorem}

The proof of Theorem \ref{t11} will lead to some strengthened forms of (ii), (iii) and (iv).
In order to state it we denote by $A_q$ and $B_q$ the elements $x$ of $\vv_q \setminus \uu_q$ whose greedy expansions $b(x,q)$ are finite and infinite, respectively, so that
\begin{equation*}
\vv_q \setminus \uu_q=A_q\cup B_q.
\end{equation*}
Given a base $q \in (1,M+1]$, we also introduce the \emph{reflection map} $\ell: J_q \to J_q$, given by
\begin{equation*}
\ell(x)= \frac{M}{q-1}-x, \quad x \in J_q.
\end{equation*}
In Part (iv) of the next proposition we refer to the expansions of $1$ in a given base $q \in \uuu\setminus \uu$.
They are listed in Theorem \ref{t32} (vi).

\begin{proposition}\label{p12}\mbox{}
Let $q\in(1,M+1]$ and write $(\alpha_i):=a(1,q)$.
\begin{enumerate}[\upshape(i)]
\item If $q=M+1$, then $\vv_q=J_q=[0,1]$, $B_q=\varnothing$, and  $A_q=\vv_q\setminus\uu_q$ is dense in $\vv_q$.
\item If $q \in \uuu\setminus\set{M+1}$, then both $A_q$ and $B_q$ are dense in $\vv_q$.
Moreover,
\begin{equation}\label{11}
B_q = \ell(A_q),
\end{equation}
and the  greedy expansion of each $x \in B_q$ ends with $\overline{(\alpha_i)}$.
\item If $q\in\uu$, then every $x\in\vv_q\setminus\uu_q$ has exactly two expansions:
\begin{enumerate}[\upshape (a)]
\item if $x \in A_q$ and
$b(x,q)=b_1\cdots b_n0^{\infty}$ with $b_n>0$,
then $(b_1\cdots b_n)^-\alpha_1\alpha_2\cdots$ is the other expansion of $x$; 
\item  if $x\in B_q$, then the expansions of $x$ are the reflections of the expansions of $\ell(x) \in A_q$.
\end{enumerate}
\item If $q\in\uuu\setminus\uu$, then every $x \in \vv_q \setminus \uu_q$  has exactly $\aleph_0$ expansions:
\begin{enumerate}[\upshape (a)]
\item if $x \in A_q$ and
$b(x,q)=b_1\cdots b_n0^{\infty}$ with $b_n>0$,
then the other expansions of $x$ are of the form
$(b_1\cdots b_n)^-d_{n+1}d_{n+2}\cdots,$
where $d_{n+1}d_{n+2}\cdots$ is one of the expansions of $1$ in base $q$;
\item if $x\in B_q$, then the expansions of $x$ are the reflections of the expansions of $\ell(x) \in A_q$.
\end{enumerate}
\end{enumerate}
\end{proposition}

In Part (iv) of our next theorem we refer to the expansions of $1$ in a given base $q \in \vv \setminus \uuu$.
They are listed in Theorem \ref{t33} (vi).
We recall that if $q\in\vv\setminus \uuu$, then there is a smallest integer $k\ge 1$ such that
$\alpha_{k+1}\alpha_{k+2}\cdots=\overline{\alpha_1\alpha_2\cdots}$; see for instance Proposition~~\ref{p31}. We also denote by $\tilde q= \tilde q (M)$ the smallest element of $\vv$ (see Theorem \ref{t33} (ii)).
We will use these notations in the statement of the following theorem.

\begin{theorem}\label{t13}
Let $q \in \vv \setminus \uuu$. 
\begin{enumerate}[\upshape (i)]
\item The sets $\uu_q$ and $\vv_q$ are closed.
\item $\abs{\vv_q \setminus \uu_q}= \aleph_0$ and $\vv_q \setminus \uu_q$ is a discrete set, dense in $\vv_q$. 
\item Each element $x \in \vv_q \setminus \uu_q$ has exactly $\aleph_0$
expansions and a finite greedy expansion.
\item Let  $x \in \vv_q \setminus \uu_q$,   and let $b_n$ be the last nonzero element of $(b_i):=b(x,q)$.
Then $x$ has exactly $\aleph_0$ other expansions of the form
\begin{equation*}
(b_1\cdots b_n)^-d_{n+1}d_{n+2}\cdots,
\end{equation*}
where $d_{n+1}d_{n+2}\cdots$ is one of the expansions of $1$ in base $q$. 
\\
Furthermore, if $q=\tilde q$, $M$ is even and $b_n \ge 2$, then $b_1 \cdots b_{n-1} (b_n-2)M^{\infty}$ is also an expansion of $x$.   
\\
Finally, if 
\begin{equation*}
n>k,\quad b_{n-k}>0\qtq{and}
b_{n-k+1}\cdots b_n= \overline{(\alpha_1\cdots\alpha_k)^-},
\end{equation*} 
then $x$ has one more expansion: 
\begin{equation*}
(b_1\cdots b_{n-k})^-M^{\infty}.
\end{equation*}
\end{enumerate}
\end{theorem}

\begin{theorem}\label{t14}
If $q \in (1, M+1] \setminus \vv$, then $\uu_q = \overline{\uu_q}= \vv_q$.
\end{theorem}

\begin{remarks}\label{r15}\mbox{}

\begin{enumerate}[\upshape (i)]
\item While $\uu$, $\uuu$ and $\vv$ are three different sets, we infer from Theorems \ref{t11}, \ref{t13} and \ref{t14} that at least two of the three sets $\uu_q$, $\uuuq$ and $\vv_q$ coincide for each $q \in (1,M+1]$.

\item The set $\vv$ has Lebesgue measure zero and Hausdorff dimension one, 
 because $\vv\setminus\uu$ is countable by  \cite{[DKL]} and $\uu$ is a Lebesgue null set of Hausdorff dimension one by \cite[Theorem 1]{[EJ]} and \cite[Theorem 1.6]{[KKL]}. Theorem \ref{t14} implies that $\uu_q = \overline{\uu_q}= \vv_q$ except for a set of bases $q \in (1,M+1]$ of Lebesgue measure zero and Hausdorff dimension equal to one.

\item If $q=M+1$, then $\uuuq=[0,1]$, and $\uu_q$ has Lebesgue measure one.
Suppose next that $q\in(1,M+1)$ is a non-integer. If $x \in \uu_q$ has unique expansion $(c_i) \not= M^{\infty}$, then there exists an index $N$ such that $(c_{N+i})=c_{N+1}c_{N+2} \cdots$ is the unique expansion of a number belonging to $[0,1)$, whence $c_{N+i} < q$ for each $i \ge 1$. It follows from Proposition 4.1 in \cite{[DVK2]} (see also \cite{[GS]}) that $\uu_q$ can be covered by countably many sets of the same Hausdorff dimension less than one. Hence the sets $\uu_q$, $\uuuq$ and $\vv_q$ have Hausdorff dimension less than one and are also nowhere dense because $\vv_q \setminus  \uu_q$ is (at most) countable and $\vv_q$ is closed. If $q \in (1,M+1)$ is an integer and $M < 2q-2$, then the unique expansion  of a number $x \in \uu_q$ has a tail  belonging to $\set{0^{\infty}, M^{\infty}} \cup \set{M-q+1,\ldots,q-1}^{\infty}$ as follows, for instance, from Lemma~\ref{l42}. Conversely, a sequence belonging to $\set{M-q+1,\ldots,q-1}^{\infty}$ with no tail equal to $(M-q+1)^{\infty}$ or $(q-1)^{\infty}$  is the (unique) expansion of an element in $\uu_q$. Straightforward alterations in the usual calculation of the Hausdorff dimension of the triadic Cantor set $C$ show that $\dim_H(\uu_q)=\log(2q-M-1)/\log(q)$, see for instance \cite{[F]}.   Hence, also in this case, $\uu_q$, $\uuuq$ and $\vv_q$ are nowhere dense and have Hausdorff dimension less than one. Finally, if $q \in (1,M+1)$ is an integer and $ M \ge 2q-2$, then $\uu_q=\set{0,M/(q-1)}$. More precise information about the Hausdorff dimension of $\uu_q$ and the behaviour of the function $q \mapsto \dim_H(\uu_q)$ on $(1,M+1]$, albeit at the cost of a much more elaborate analysis,  can be found in \cite{[KKL]} and \cite{[AK]}. 
\end{enumerate}
\end{remarks}

We mention two corollaries of Theorems \ref{t11}, \ref{t13} and \ref{t14}:

\begin{corollary}\label{c16}
Let $q\in(1,M+1]$.
Then $\uu_q$ is closed if and only if $q\notin\uuu$. 
\end{corollary}

\begin{corollary}\label{c17}
Let $q \in (1,M+1]$. The following equivalences hold:
\begin{align*}
&q\in\uu\Longleftrightarrow 1\in\uu_q,\\
&q\in\uuu\Longleftrightarrow 1\in\uuuq,\\
&q\in\vv\Longleftrightarrow 1\in\vv_q.
\end{align*}
\end{corollary}

Let us denote by $\uu_q'$ and $\vv_q'$ the sets of quasi-greedy expansions in base $q$ of all numbers $x\in\uu_q$ and $x\in\vv_q$, respectively. 
Note that $\uu_q'$ is simply the set of unique expansions in base $q$. Elements of $\uu_q'$ are frequently referred to as \emph{univoque sequences (in base $q$)}. 

If we endow the collection of subsets of $A^{\NN}$ with the partial order $\subseteq$, then the maps $q\mapsto\uu'_q$ and $q\mapsto\vv'_q$ are non-decreasing by Lemma \ref{l44}.
However, on some intervals these maps are constant.
We say that an  interval $I  \subseteq (1,M+1]$ of bases is a \emph{stability interval} for $\uu_q$ (resp. for $\vv_q$) if $\uu_q'= \uu_r'$ (resp. $\vv_q'= \vv_r'$) for all $q,r\in I$. We call a stability interval $I$ \emph{maximal} if any interval $J \subseteq (1,M+1]$ that properly contains $I$ is not a stability interval. 
The following theorem completes the investigation of stability intervals started by Dar\'oczy and K\'atai (\cite{[DK1],[DK2]}):

\begin{theorem}\label{t18}\mbox{}

\begin{enumerate}[\upshape (i)]
\item The maximal stability intervals for $\uu_q$ are given by the singletons $\set{q}$ where $q \in \uuu$, and the intervals $(q_1,q_2]$ where
$(q_1,q_2)$ is a connected component of $(1,M+1]\setminus\vv$.
\item
The maximal stability intervals for $\vv_q$ are given by the singletons $\set{q}$ where $q \in \uu$, the interval $(1,\tilde q)$, and the intervals $[q_1,q_2)$ where
$(q_1,q_2)$ is a connected component of $(1,M+1]\setminus\vv$ with $q_1 \not=1$.
\end{enumerate}
\end{theorem}

In \cite{[KL3], [DVK1], [DKL]} we have clarified the topological structure of the complements $(1, M+1] \setminus\uuu$ and $(1,M+1] \setminus\vv$.
The following theorem describes the topological structure of $J_q\setminus\uuuq$ and $J_q\setminus\vv_q$ for all $q\in(1,M+1]$.
We recall that $\uu_q=\uuuq=\vv_q$ if $q\notin\vv$, and that $\uu_q  \subsetneq\vv_q$ if $q\in\vv$. Moreover, if we write the open set $(1,M+1] \setminus \vv$ as a disjoint union of open intervals, then the set of left and right endpoint of these intervals are given by $\set{1} \cup \left(\vv \setminus \uu \right)$ and $\vv \setminus \uuu$, respectively; see Theorem~\ref{t33} (iii).

\begin{theorem}\label{t19}\mbox{}
Let $q\in(1,M+1]$ and write $(\alpha_i):=a(1,q)$.
\begin{enumerate}[\upshape (i)]
\item If $q=M+1$, then $\uuuq=\vv_q=J_q=[0,1]$.
\item If $q\in\uuu\setminus\set{M+1}$, then 
$J_q\setminus\overline{\uu_q}=J_q\setminus\vv_q$
is the union of infinitely many disjoint open sets $(x_L,x_R)$.
Furthermore,  $x_L$ and $x_R$ run over $A_q$ and $B_q$, respectively. More precisely, if 
$b(x_L,q)=b_1\cdots b_n0^{\infty}$ with $b_n>0$, then $b(x_R,q)=b_1\cdots b_n\overline{\alpha_1 \alpha_2 \cdots}$.
\item If $q\in\vv\setminus\uuu$, then
$J_q\setminus\overline{\uu_q}=J_q\setminus\uu_q$
is an open set.
Furthermore, each connected component $(x_L,x_R)$ of $J_q\setminus\uu_q$ contains infinitely many elements of $\vv_q$, forming an increasing sequence $(x_k)_{k=-\infty}^{\infty}$ satisfying 
\begin{equation}\label{12}
x_k\to x_L \qtq{as}k\to-\infty,\quad
x_k\to x_R \qtq{as}k\to \infty.
\end{equation}
Moreover, the relation between two subsequent numbers $x_{m}$ and $x_{m+1}$ in the sequence $(x_k)_{k=-\infty}^{\infty}$ is described as follows: 
\begin{equation}\label{subsequent}
\text{if } b(x_{m},q)=b_1 \cdots b_n 0^{\infty} \qtq{with} b_n>0, \qtq{then} a(x_{m+1},q)=b_1\cdots b_n\overline{\alpha_1 \alpha_2 \cdots}. 
\end{equation}
\item If $q\in(1,\tilde q]$, then
$J_q\setminus\uu_q=(0,M/(q-1))$.
\item Let $(q_1,q_2)$ be a connected component of $(1,M+1]\setminus\vv$ with 
$q_1 \not= 1$ and $q\in(q_1,q_2]$.
Then $\uu_q'=\vv_{q_1}'$ and the open sets $J_q \setminus \uu_q$ and $J_{q_1} \setminus \vv_{q_1}$ are homeomorphic.
\end{enumerate}
\end{theorem}

\begin{remarks}\label{r110}\mbox{}

\begin{enumerate}[\upshape (i)]
\item
If $q\in\vv\setminus\uuu$ and $q\ne\tilde q$, then  $J_q\setminus\uu_q$ has infinitely many connected components.
Indeed, there exists a connected component $(q_1,q_2)$ of $(1,M+1)\setminus\vv$ such that $q_1\in\vv\setminus\uu$ and $q=q_2$. 
Then $J_q\setminus\uu_q$ is homeomorphic to $J_{q_1}\setminus\vv_{q_1}$ by Theorem \ref{t19} (v), and $J_{q_1}\setminus\vv_{q_1}$ has infinitely many connected components by Theorem \ref{t19} (ii) and (iii).
\item If $q \in \vv \setminus \uuu$ and $(x_L,x_R)$ is a connected component of $J_q \setminus \uu_q$, then all numbers in $(x_L,x_R) \cap \vv_q$ have a finite greedy expansion, according to Theorem~\ref{t13} (iii).   The greedy expansion $b(x_{m+1},q)$ of the number $x_{m+1}$ occuring in Theorem~\ref{t19} (iii) equals $\left(a_1(x_{m+1},q) \cdots a_k(x_{m+1},q)\right)^+ 0^{\infty}$ where $k \ge 1$ is the smallest index such that $a_k (x_{m+1},q) < M$ and $(a_{k+i}(x_{m+1},q))=(\alpha_i)$; see Lemma~\ref{l28} (ii).    
\item For $q \in (1,\tilde q)$, we provide in Examples \ref{e45} (ii) a short proof of the following strengthening of Theorem ~\ref{t19} (iv): \emph{each} $x \in \left(0,M/(q-1)\right)$ has $2^{\aleph_0}$ expansions. This result was first established by Baker (\cite{[B]}) using different ideas. We extract from our proof some other results of Baker (\cite{[B]}) in Examples \ref{e45} (iii) and (iv).   
\end{enumerate}
\end{remarks}

We recall that a  nonempty closed set is called \emph{perfect} if it has no isolated points, and that $x$ is a \emph{condensation point} of a set $F  \subseteq \RR$ if every neighborhood of $x$ contains uncountably many elements of $F$. 
By the Cantor--Bendixson theorem, the condensation points of an uncountable closed set $F$ of real numbers form a perfect set $P$, and $F\setminus P$ is 
(at most) countable.
In the following theorem we determine the 
isolated, accumulation and condensation points of the sets $\uu_q$, $\uuuq$ and $\vv_q$, and we list explicitly all cases where $\uu_q$, $\uuuq$ and $\vv_q$ is a Cantor set.
We recall from Theorems~\ref{t32} (iv) and \ref{t33} (iii) that if  $(q_0,q_0^*)$ is a connected component of $(1,M+1] \setminus \uuu$, then $q_0 \in \set{1} \cup \left( \uuu \setminus \uu \right)$,  $q_0^* \in \uu$, and $\vv\cap (q_0,q_0^*)$ is formed by an increasing sequence $q_1<q_2<\cdots,$ converging to $q_0^*$.
We use this notation in the following theorem. 
We also recall that if $q_0=1$, then $q_1=\tilde q=\tilde q(M)$ and $q_0^*$ is the \emph{Komornik--Loreti constant} which we will denote by $q_{KL}=q_{KL}(M)$; see \cite{[KL1],[KL2]}. The precise value of $q_{KL}(M)$ is given in Theorem~\ref{t32} (ii).

\begin{theorem}\label{t111}\mbox{}
\begin{enumerate}[\upshape (i)]
\item Let $q\in\uuu$.
\begin{enumerate}[\upshape (a)]
\item If $q=M+1$, then $\vv_q=\uuuq=[0,1]$ is a perfect set.
\item If $q\in\uuu\setminus\set{M+1}$, then $\vv_q=\uuuq$ is a Cantor set.
\end{enumerate}
\item Consider the connected component $(1,q_{KL})$ of $(1,M+1] \setminus\uuu$. 
\begin{enumerate}[\upshape (a)]
\item If $q\in(1,q_1]$, then $\uu_q$ is a two-point set.
\item If $q\in(q_n,q_{n+1}]$ for some $n\ge 1$, then $\uu_q$ is countably infinite. 
Furthermore, its accumulation and isolated points form  the sets
\begin{equation*}
\pi_q(\uu'_{q_n})\qtq{and}
\pi_q(\vv'_{q_n}\setminus\uu'_{q_n}),
\end{equation*}
respectively, and the isolated points of  $\uu_q$ form a dense subset of $\uu_q$. 
\end{enumerate}
\item Consider a connected component $(q_0,q_0^*)$ of $(1,M+1]\setminus\uuu$ with $q_0\in\uuu\setminus\uu$.
\begin{enumerate}[\upshape (a)]
\item If $q\in(q_0,q_1]$, then $\uu_q=\uuuq$ is a Cantor set, and  $\pi_q(\vv'_{q_0}\setminus\uu'_{q_0})$ is dense in $\uu_q$.
\item If $q\in(q_n,q_{n+1}]$ for some $n\ge 1$, then the condensation points, further accumulation points and isolated points of $\uu_q$ form  the sets
\begin{equation*}
\pi_q(\vv'_{q_0}),\quad 
\pi_q(\uu'_{q_n}\setminus\vv'_{q_0})\qtq{and}
\pi_q(\vv'_{q_n}\setminus\uu'_{q_n}),
\end{equation*}
respectively, and the isolated points of $\uu_q$ form a dense subset of $\uu_q$. 
\end{enumerate}
\end{enumerate}
\end{theorem}

We recall (see, e.g., \cite{[LM]}) that a set $S \subseteq A^{\NN}$ is called a \emph{shiftspace} or \emph{shift} if
there exists a set $\mathcal{F}(S)  \subseteq \cup_{k=1}^{\infty}
A^{k}$ such that a sequence $(c_i) \in
A^{\NN}$ belongs to $S$ if and only if none of the blocks $c_{i+1} \cdots c_{i+n}$ $(i \ge 0, n \ge 1)$ belongs to $\mathcal{F}(S)$. A shift $S$ is called a
\emph{shift of finite type} if one can choose $\mathcal{F}(S)$ to be finite. We endow the alphabet $A$ with the discrete topology, and the set of expansions $A^{\NN}$ with the Tychonov product topology, so that the corresponding convergence in $A^{\NN}$ is the coordinate-wise convergence of sequences. 

\begin{theorem}\label{t112}
Let $q >1$ be a real number. 
The following statements are equivalent.
\begin{enumerate}[\upshape (i)]
\item $q \in (1, M+1] \setminus \uuu$.
\item $\uu_q'$ is a shift of finite type.
\item $\uu_q'$ is a shift.
\item $\uu_q'$ is a closed subset of $A^{\NN}$.
\end{enumerate}
\end{theorem}

Finally we consider the \emph{two-dimensional univoque set} $\bfu $, formed by the couples $(x,q)\in\RR^2$ where $q>1$ and $x$ has a unique expansion in base $q$ over the alphabet $A$:
\begin{equation*}
\bfu:=\set{(x,q)\in\RR^2\ :\ q\in(1,M+1]\text{ and }x\in\uu_q}.
\end{equation*}
Setting 
\begin{equation*}
\bfv:=\set{(x,q)\in\RR^2\ :\ q\in(1,M+1]\text{ and }x\in\vv_q},
\end{equation*}
we have the following result:

\begin{theorem}\label{t113}
The set $\bfu $ is not closed. 
Its closure $\bfuu$  equals $\bfv \cup \set{(0,1)}$. 
\end{theorem}
Hence, $\bfuu \cap \bfj = \bfv$, where the set $$\bfj:=\set{(x,q): q \in (1,M+1] \text{ and } x \in J_q}$$ consists of all couples $(x,q)$ such that $x$ has an expansion in base $q$. 

The rest of the paper is organised as follows.
In Sections \ref{s2} and \ref{s3} we recall various properties of the greedy and quasi-greedy expansions and of the sets of bases $\uu$, $\uuu$ and $\vv$.
In Section \ref{s4} we deduce some elementary properties of the sets $\uu_q$ and $\vv_q$ and we reprove the main results in \cite{[B]}. In Section \ref{s5} we prove two density results that will be important in the proof of our main theorems.
In Section \ref{s6} we prove Theorem \ref{t11} and Proposition \ref{p12}.
Section \ref{s7} is devoted to the proof of Theorems \ref{t13}, \ref{t14} and Corollaries \ref{c16}, \ref{c17}.
We also give an intrinsic characterization of $\uuuq$ in Theorem \ref{t72}.
Our final Theorems \ref{t18}, \ref{t19}, \ref{t111}, \ref{t112} and \ref{t113} are proved in Section \ref{s8}.
For the reader's convenience, a list of principal terminology and notations used in this paper is given in the final Section \ref{s9}.

\section{Greedy and quasi-greedy expansions}\label{s2}

In this section we recall from \cite{[BK], [KL2], [DVK2], [DKL]} some results that we shall use very frequently in the sequel. 

Given a base $q \in (1,M+1]$ and a real number $x\ge 0$, we define the {\it greedy sequence} $b(x,q)=(b_i(x,q))$ and the {\it quasi-greedy sequence} $a(x,q)=(a_i(x,q))$ in $A^{\NN}$ as follows.
If for some $n\in\NN$, $b_i(x,q)$ is already defined for every $i<n$ (no condition if $n=1$), then let $b_n(x,q)$ be the largest element of the digit set $A$ such that
\begin{equation*}
\sum_{i=1}^n\frac{b_i(x,q)}{q^i}\le x.
\end{equation*} 
Similarly, if $x>0$ and if for some $n\in\NN$, $a_i(x,q)$ is already defined for every $i<n$ (no condition if $n=1$), then let $a_n(x,q)$ be the largest element of the digit set $A$ such that
\begin{equation*}
\sum_{i=1}^n\frac{a_i(x,q)}{q^i}<x.
\end{equation*}
Furthermore, we set $a(0,q):=0^{\infty}$.
It follows from the definitions that
\begin{equation*}
\sum_{i=1}^{\infty}\frac{a_i(x,q)}{q^i}
\le\sum_{i=1}^{\infty}\frac{b_i(x,q)}{q^i}\le x
\end{equation*} 
for all $x\ge 0$.
Moreover, $(b_i(x,q))$ is the lexicographically largest sequence in $A^{\NN}$ satisfying
\begin{equation*}
\sum_{i=1}^{\infty}\frac{b_i(x,q)}{q^i}\le x,
\end{equation*}
and $(a_i(x,q))$ is the lexicographically largest infinite sequence in $A^{\NN}$ satisfying
\begin{equation*}
\sum_{i=1}^{\infty}\frac{a_i(x,q)}{q^i}\le x.
\end{equation*}

If $x \in J_q$, then the sequences $a(x,q)$ and $b(x,q)$ are indeed expansions of $x$ and are thus the quasi-greedy and greedy expansion of $x$ respectively, as introduced in Section ~\ref{s1}:

\begin{proposition}\label{p22}\mbox{}
Let $(x,q) \in \bfj$.
Then
\begin{enumerate}[\upshape (i)]
\item $a(x,q)$ and $b(x,q)$ are expansions of $x$, i.e.,
\begin{equation*}
\sum_{i=1}^{\infty}\frac{a_i(x,q)}{q^i}
=\sum_{i=1}^{\infty}\frac{b_i(x,q)}{q^i}
=x.
\end{equation*}
\item If $q\not= M+1$, then $a(x,q)$ is doubly infinite.
\end{enumerate}
\end{proposition}

The expansions $b(x,q)$ and $a(x,q)$ have the following  semi-continuity properties:

\begin{lemma}\label{l21}
Let $(x,q)\in \bfj$ and $(y_n,r_n) \in \bfj$, $n \in \NN$.
Then
\begin{enumerate}[\upshape (i)]
\item for each positive integer $m$
there exists a neighborhood $W \subseteq \RR^2$ of $(x,q)$ such that
\begin{equation}\label{21}
b_1(y,r)\cdots b_m(y,r)\le b_1(x,q)\cdots b_m(x,q) \qtq{for all}(y,r)\in W\cap \bfj;
\end{equation}
\item if  $y_n \downarrow x$ and $r_n \downarrow q$, then
$b(y_n,r_n)$ converges (coordinate-wise) to $b(x,q)$.
\item for each positive integer $m$
there exists a neighborhood $W  \subseteq \RR^2$ of $(x,q)$ such that
\begin{equation}\label{22}
a_1(y,r)\cdots a_m(y,r)\ge a_1(x,q)\cdots a_m(x,q) \qtq{for all}(y,r)\in W \cap \bfj;
\end{equation}
\item if  $y_n \uparrow x$ and $r_n \uparrow q$, then
$a(y_n,r_n)$ converges (coordinate-wise) to $a(x,q)$.
\end{enumerate}
\end{lemma}

\begin{proof}
(i) By definition of $b(x,q)$, we have the \emph{strict} inequalities 
\begin{equation*}
\sum_{i=1}^n\frac{b_i(x,q)}{q^i}>x-\frac{1}{q^n}\qtq{whenever} b_n(x,q)<M.
\end{equation*}
Hence, if $(y,r)\in \bfj$ is sufficiently close to $(x,q)$, then the \emph{finitely} many inequalities   
\begin{equation*}
\sum_{i=1}^n\frac{b_i(x,q)}{r^i}>y-\frac{1}{r^n} : n\le m \text{ and }b_n(x,q)< M, 
\end{equation*}
also hold. These inequalities imply \eqref{21}. Note that the inequality \eqref{21} can be strict even if $(y,r)$ is very close to $(x,q)$ because it may be the case that 
\begin{equation*}
\sum_{i=1}^n \frac{b_i(x,q)}{ r^i} > y
\end{equation*}
for some $n \le m$. 
\medskip

(ii) If $y_n\ge x$ and $r_n\ge q$, we deduce from the definition of greedy sequences that 
\begin{equation*}
b(x,q)\le b(y_n,r_n)
\end{equation*}
for every $n$. Equivalently, we have
\begin{equation*}
b_1(x,q)\cdots b_m(x,q)\le b_1(y_n,r_n)\cdots b_m(y_n,r_n)
\end{equation*}
for all positive integers $m$ and $n$. It remains to notice that by the previous part the converse inequality also
holds for each fixed $m$ if $n$ is large enough.
\medskip

(iii) We may assume that $x \not=0$. By definition we have
\begin{equation*}
\sum_{i=1}^n\frac{a_i(x,q)}{q^i}<x \text{ for all } n=1,2,\ldots.
\end{equation*}
If $(y,r)\in \bfj$ is sufficiently close to $(x,q)$, then
\begin{equation*}
\sum_{i=1}^n\frac{a_i(x,q)}{r^i}<y,\quad n=1,\ldots, m.
\end{equation*}
These inequalities imply \eqref{22}.
\medskip

(iv) If $y_n\le x$ and $r_n\le q$, we deduce from the definition of quasi-greedy sequences that
\begin{equation*}
a(x,q)\ge a(y_n,r_n)
\end{equation*}
for every $n$. Equivalently, we have
\begin{equation*}
a_1(x,q)\cdots a_m(x,q)\ge a_1(y_n,r_n)\cdots a_m(y_n,r_n)
\end{equation*}
for all positive integers $m$ and $n$. It remains to notice that by the previous part the converse inequality also
holds for each fixed $m$ if $n$ is large enough.
\end{proof}

Let $q\in(1,M+1]$.
Observe that for any fixed  $x \in J_q$, the map $(c_i)\mapsto (M-c_i)$ is a strictly decreasing bijection between the expansions of $x$ and $\ell(x)$.
Therefore every $x \in J_q$ has a lexicographically smallest expansion, given by $(M-b_i)$, where $(b_i)=b(\ell(x),q)$. We often call this expansion the \emph{lazy} expansion of $x$. 

\begin{lemma}\label{l25}
If $q\in(1,M+1]$, then the lazy expansion of every $x \in J_q$ is infinite.
\end{lemma}

\begin{proof}
Write $(\alpha_i):=a(1,q)$.
If an expansion $(c_i)$ of $x \in J_q$ has a last nonzero digit $c_n$, then
$(c_1\cdots c_n)^-\alpha_1\alpha_2\cdots$ is another  expansion of $x$, smaller than $(c_i)$.
\end{proof}

Since we will often use the greedy and quasi-greedy expansions of $1$, we introduce for brevity the notation
\begin{equation*}
\beta(q)=(\beta_i(q)):= b(1,q) \qtq{and} 
\alpha(q)=(\alpha_i(q)):=a(1,q).
\end{equation*}
We will often write $(\beta_i)$ and $(\alpha_i)$ when the base $q$ is known from the context.
It will also be convenient to define 
\begin{equation*}
\beta(1)=(\beta_i(1)):=10^{\infty}\qtq{and}
\alpha(1)=(\alpha_i(1)):=0^{\infty}. 
\end{equation*}

\begin{proposition}\label{p26}\mbox{}
\begin{enumerate}[\upshape (i)]
\item The map $q \mapsto \beta(q)$ is a strictly increasing bijection from
 $[1,M+1]$ onto the set of all sequences $(\beta_i)$,
satisfying
\begin{equation}\label{23}
\beta_{n+1} \beta_{n+2} \cdots < \beta_1 \beta_2 \cdots \qtq{whenever}\beta_n<M.
\end{equation}
Furthermore, in case $q\in[1,M+1)$ the inequality \eqref{23} holds for all $n\geq 1$.
\item The map $q \mapsto \alpha(q)$ is a strictly increasing bijection from
 $[1,M+1]$ onto the set of all {\rm infinite} sequences $(\alpha_i)$,
satisfying
\begin{equation}\label{24}
\alpha_{n+1} \alpha_{n+2} \cdots \leq \alpha_1 \alpha_2 \cdots \qtq{whenever}\alpha_n<M.
\end{equation}
Furthermore, the inequality \eqref{24} holds for all $n\geq 0$.
\item $\alpha(q)$ is doubly infinite for every $q\in[1,M+1]$.
\end{enumerate}
\end{proposition}

\begin{proposition}\label{p27}\mbox{}
Let $q\in(1,M+1]$, and write $(\alpha_i)=\alpha(q)$.
\begin{enumerate}[\upshape (i)]
\item The map  $x \mapsto b(x,q)$ is a strictly increasing bijection
from $J_q$ onto the set of all sequences $(b_i)$, satisfying
\begin{equation}\label{25}
b_{n+1}b_{n+2} \cdots < \alpha_1 \alpha_2 \cdots \qtq{whenever}b_n < M.
\end{equation}
Furthermore, the inequality \eqref{25} holds whenever $b_1\cdots b_n\ne M^n$.

\item The map $x \mapsto a(x,q)$ is a strictly increasing bijection
from $J_q$ onto the set of all {\rm infinite} sequences $(a_i)$,
satisfying
\begin{equation}\label{26}
a_{n+1}a_{n+2}\cdots \leq \alpha_1 \alpha_2 \cdots \qtq{whenever}a_n < M.
\end{equation}
Furthermore, the inequality \eqref{26} holds whenever $a_1\cdots a_n\ne M^n$.
\end{enumerate}
\end{proposition}

We call a sequence $(c_i)$ {\it periodic} if $(c_i)=(c_{k+i})=c_{k+1}c_{k+2} \cdots$ for some $k \ge 1$. The smallest positive integer $k$ for which a periodic sequence satisfies $(c_{k+i})=(c_i)$ is called the {\it smallest period} of $(c_i)$. 
\begin{lemma}\label{l28}
Let $(x,q) \in \bfj$, and set
\begin{equation*}
(b_i)=b(x,q),\quad 
(a_i)=a(x,q),\quad 
(\beta_i)=\beta(q),\quad 
(\alpha_i)=\alpha(q).
\end{equation*}

\begin{enumerate}[\upshape (i)]
\item If $b(x,q)$ is infinite, then $a(x,q)=b(x,q)$.
\item If $(b_i)$ has a last nonzero element $b_m$, then
\begin{equation*}
(a_i)=(b_1\cdots  b_m)^-\alpha_1\alpha_2\cdots .
\end{equation*}

\item If $(\beta_i)$ is infinite, then $(\beta_i)=(\alpha_i)$ is periodic only for $q=M+1$.

\item If $(\beta_i)$ has a last nonzero element $\beta_m$, then 
\begin{equation*}
(\alpha_i)=((\beta_1\cdots \beta_m)^-)^\infty,
\end{equation*}
and $m$ is the smallest period of $(\alpha_i)$.
\end{enumerate}
\end{lemma}

\begin{lemma}\label{l29}
Let $q\in(1,M+1]$, and let $(d_i)=d_1d_2 \cdots$ be a greedy or quasi-greedy sequence.
Then for all $N \geq 1$ the truncated
sequence $d_1 \cdots d_N 0^{\infty}$ is greedy.
\end{lemma}

\begin{proof}
The statement follows at once from
Proposition~\ref{p27}.
\end{proof}

\begin{lemma}\label{l210}
Let $q\in(1,M+1]$, and write $(\alpha_i)=\alpha(q)$.
Let $(b_i) \ne M^{\infty}$ be a greedy sequence.
Then

\begin{enumerate}[\upshape (i)]
\item There
exists a sequence $1 \leq n_1 < n_2 < \cdots$ such that for each
$i \geq 1$,
\begin{equation*}
b_{n_i} < M, \qtq{and}
b_{m+1} \cdots b_{n_i} < \alpha_1 \cdots \alpha_{n_i-m} \qtq{if}  1\le m <n_i  \qtq{and}  b_m < M.
\end{equation*}
\item For every positive integer $N$, there exists a greedy sequence 
$(c_i)>(b_i)$ such
that
\begin{equation*}
c_1 \cdots c_N= b_1\cdots b_N.
\end{equation*}
\end{enumerate}
\end{lemma}

\begin{proof}
(i) See \cite[Theorem 2.1]{[DKL]}.
\medskip

(ii) This is a consequence of Lemma \ref{l21} (ii).
\end{proof}

\section{A review of the sets of bases $\uu$, $\uuu$ and $\vv$}\label{s3}

In this section we recall several results
from \cite{[DKL]}.
They generalized for $1<q\le M+1$ a number of theorems proved in \cite{[KL3]} for $M<q\le M+1$.

The three sets $\uu$, $\uuu$ and $\vv$ (which are clearly all contained in $(1,M+1]$) have the following lexicographic characterizations (the first one was proved in \cite{[EJK]} for $M=1$):

\begin{proposition}\label{p31}
Let $q\in (1,M+1]$, and write $(\alpha_i)=\alpha(q)$, $(\beta_i)=\beta(q)$.
We have
\begin{enumerate}[\upshape (i)]
\item $q\in\uu\Longleftrightarrow \overline{(\beta_{n+i})}<(\beta_i)$ whenever $\beta_n>0$;
\item $q\in\uuu\Longleftrightarrow \overline{(\alpha_{n+i})}<(\alpha_i)$ whenever $\alpha_n>0$;
\item $q\in\vv\Longleftrightarrow \overline{(\alpha_{n+i})}\le(\alpha_i)$ whenever $\alpha_n>0$.
\end{enumerate}
Moreover, in each case the inequalities are satisfied for all $n\ge 0$.
\end{proposition}

The main properties of $\uu$ and $\uuu$ are contained in the following theorem.
In its statement we use the Thue--Morse sequence  $(\tau_i)_{i=0}^{\infty}$,  defined by the formulas $\tau_0:=0$, and
\begin{equation*}
\tau_{2^N+i}=1-\tau_{i}\quad\text{for}\quad i=0,\ldots, 2^N-1,\quad N=0,1,2,\ldots .
\end{equation*}
We call a set $X\subseteq\RR$ \emph{closed from above (below)} if the limit of every bounded decreasing (increasing) sequence in $X$ belongs to $X$.
Alternatively, $X\subseteq\RR$ is closed from above (below) if for each $x \in \RR \setminus X$, there exists a $\delta =\delta(x) >0$ such that
\begin{equation*}
\left[x,x+\delta\right) \cap X = \varnothing
\quad (\left(x-\delta,x \right] \cap X=\varnothing).
\end{equation*} 

\begin{theorem}\label{t32}\mbox{}

\begin{enumerate}[\upshape (i)]
\item The set $\uu$ is closed from above but not from below.   
\item The smallest element $q_{KL}=q_{KL}(M)$ of $\uu$ (called \emph{Komornik--Loreti constant}) is determined  by the formula
\begin{equation*}
\alpha_i(q')=
\begin{cases}
m-1+\tau_i&\text{if $M=2m-1$,\quad $m=1,2,\ldots ;$}\\
m+\tau_i-\tau_{i-1}&\text{if $M=2m$,\quad $m=1,2,\ldots .$}
\end{cases}
\end{equation*}
\item The closure $\uuu$ of $\uu$ is a Cantor set.
Moreover, $\uuu \setminus  \uu$ is a countable dense set in $\uuu $.
\item We have a disjoint union
\begin{equation*}
(1,M+1]\setminus \uuu=\cup^* (q_0,q_0^*)
\end{equation*}
where $q_0$ runs over $\set{1}\cup (\uuu\setminus\uu)$ and $q_0^*$ runs over a proper subset $\uu^*$ of $\uu$.
\item If $q\in\uuu$ and $(\alpha_i)=\alpha(q)$, then there exist arbitrarily large integers $m$ such
that for all $k$ with $0 \le k < m$,
\begin{equation*}
\overline{\alpha_{k+1} \cdots \alpha_{m}} < \alpha_1 \cdots \alpha_{m-k}.
\end{equation*}

\item If $q\in\uuu\setminus\uu$, then $(\alpha_i)=\alpha(q)$ is periodic.
Furthermore, all expansions of $1$ are given by $(\alpha_i)$ and the sequences
\begin{equation*}
(\alpha_1\cdots\alpha_k)^N(\alpha_1\cdots \alpha_k)^+0^\infty,\quad N=0,1,\ldots,
\end{equation*}
where $k$ is the smallest period of $(\alpha_i)$.
\end{enumerate}
\end{theorem}

Next we list the main properties of the set $\vv$:

\begin{theorem}\label{t33}\mbox{}
\begin{enumerate}[\upshape (i)]
\item $\vv$ is compact, and $\vv\setminus \uuu $ is a countable dense subset of $\vv$.
\item  The smallest element $\tilde q = \tilde q (M)$ of $\vv$ (called  generalized Golden ratio) is given by the formulas
\begin{equation*}
\tilde q=
\begin{cases}
m+1&\text{if $M=2m$,}\\
\left(m+\sqrt{m^2+4m}\right)/2&\text{if $M=2m-1$}
\end{cases}
\end{equation*}
for $m=1,2,\ldots .$
Furthermore, 
\begin{equation*}
\begin{cases}
\beta(\tilde q)=m^+0^{\infty}\qtq{and}\alpha(\tilde q)=m^{\infty}&\text{if $M=2m$,}\\
\beta(\tilde q)=mm0^{\infty}\qtq{and}\alpha(\tilde q)=(mm^-)^{\infty}&\text{if $M=2m-1$.}
\end{cases}
\end{equation*}
\item The set $\vv\setminus \uuu $ is  discrete.
Moreover,  we have
\begin{equation*}
\vv\cap (q_0,q_0^*)=\set{q_n\ :\ n=1,2,\ldots}
\end{equation*}
for each connected component $(q_0,q_0^*)$ of $(1,M+1]\setminus \uuu$, where $(q_n)$ is a (strictly) increasing sequence converging to $q_0^*$. It follows from Theorem~\ref{t32} (iv) that we have a disjoint union 
\begin{equation*}
(1,M+1]\setminus \vv=\cup^*  (r_0,r_0^*)
\end{equation*}
where $r_0$ runs over $\set{1}\cup (\vv\setminus\uu)$ and $r_0^*$ runs over $\vv \setminus \uuu$.
 
\item Given $q_0\in \set{1}\cup(\uuu\setminus\uu)$, the greedy expansions of the numbers $q_n$ in (iii) have the form $\beta(q_n)=s_n0^{\infty}$ with a sequence of words $s_n$ defined recursively  as follows.
If $q_0\in \uuu\setminus\uu$, then $(\beta_i):=\beta(q_0)$ has a last nonzero digit $\beta_m$, and we define
\begin{equation*}
s_0:=\beta_1\cdots\beta_m,\qtq{and}
s_{n+1}:=s_n\overline{s_n^-},\quad n=0,1,\ldots.
\end{equation*}
If $q_0=1$, then $q_1=\tilde q$, and we define $s_0:=1$, 
\begin{equation*}
s_1:=
\begin{cases}
m^+&\text{if $M=2m$,}\\
mm&\text{if $M=2m-1$,}
\end{cases}
\end{equation*}
and
\begin{equation*}
s_{n+1}:=s_n\overline{s_n^-},\quad n=1,2,\ldots.
\end{equation*}
\item Let $q\in\vv$ and $(\alpha_i)=\alpha(q)$.
If for some $k\ge 1$,
\begin{equation}\label{32}
\overline{\alpha_{k+1}\cdots \alpha_{2k}} = \alpha_1 \cdots \alpha_k,
\end{equation}
then
\begin{equation*}
\alpha_k>0 \qtq{and} (\alpha_i)= (\alpha_1 \cdots \alpha_k \overline{\alpha_1 \cdots \alpha_k})^{\infty}.
\end{equation*}
In particular, $q\in\vv\setminus\uuu$.
Let, moreover, $n$ be the smallest index $k$ in \eqref{32}.
Then $2n$ is the smallest period of $(\alpha_i)$, except if $M=2m$ is even and $q=m+1$ (see (ii)).

\item If $q\in\vv\backslash\overline{\uu}$, then $(\alpha_i)=\alpha(q)$ is periodic. 

If $M=2m$ is even and $q=m+1$, then all $q$-expansions are given by $(\alpha_i)=m^{\infty}$ and the sequences
\begin{equation*}
m^Nm^+0^{\infty}\qtq{and} m^Nm^-M^{\infty},\quad  N=0,1,\ldots .
\end{equation*}

Otherwise all $q$-expansions are given by $(\alpha_i)$ and the sequences
\begin{equation*}
(\alpha_1\cdots\alpha_{2n})^N(\alpha_1\cdots \alpha_{2n})^+ 0^\infty,\quad  N=0,1,\ldots
\end{equation*}
and
\begin{equation*}
(\alpha_1\cdots\alpha_{2n})^N(\alpha_1\cdots \alpha_n)^- M^\infty,\quad N=0,1,\ldots,
\end{equation*}
where $2n$ is the smallest period of $(\alpha_i)$ (see (v)). 
\end{enumerate}
\end{theorem}

\section{Preliminary results on $\uu_q$ and $\vv_q$}\label{s4}

Given a base $q\in (1,M+1]$, there is a very useful lexicographic characterization of $\uu'_q$, essentially obtained in \cite{[P], [EJK],[KL2]}:

\begin{lemma}\label{l42}
Let $q\in (1,M+1]$ and $(\alpha_i)=\alpha(q)$.
A sequence $(c_i)\in A^{\NN}$ belongs to $\uu'_q$ if and only if the following two conditions are satisfied:
\begin{align}
&c_{n+1}c_{n+2} \cdots < \alpha_1 \alpha_2 \cdots \qtq{whenever}c_n < M,\label{41}\\
&\overline{c_{n+1}c_{n+2} \cdots }< \alpha_1 \alpha_2 \cdots \qtq{whenever}c_n>0.\label{42}
\end{align}
Furthermore, these inequalities hold whenever $c_1\cdots c_n\ne M^n$ and $c_1\cdots c_n\ne 0^n$, respectively.
\end{lemma}

\begin{proof}
The sequence $(c_i)$ is the unique expansion of a number $x\in J_q$ if and only if it is both the greedy and the lazy expansion of $x$.
Moreover, $(c_i)$ is the lazy expansion of $x$ if and only if $\overline{(c_i)}$ is the greedy expansion of $\ell(x)$.
Proposition \ref{p27} (i) yields at once the desired characterization of $\uu_q'$.
\end{proof}
Note that sequences satisfying \eqref{41} and \eqref{42} are always doubly infinite. 
 
Motivated by Lemma \ref{l42}, it will be convenient to change our definition of $\vv_q'$ and $\vv_q$:

\begin{definition}\label{d43}
Let $q\in (1,M+1]$ and $\alpha(q)=(\alpha_i)$.
We denote by $\vv'_q$ the set of \emph{infinite} sequences $(c_i)\in A^{\NN}$ satisfying 
the following two conditions:
\begin{align}
&c_{n+1}c_{n+2} \cdots \le \alpha_1 \alpha_2 \cdots \qtq{whenever}c_n < M,\label{43}\\
&\overline{c_{n+1}c_{n+2} \cdots }\le \alpha_1 \alpha_2 \cdots \qtq{whenever}c_n>0,\label{44}
\end{align}
and we define
\begin{equation*}
\vv_q:=\set{\pi_q((c_i))\ :\ (c_i)\in\vv'_q}.
\end{equation*}
\end{definition}
We will show later in Theorem \ref{t72} that Definition \ref{d43} is equivalent to the former one, given in the introduction.
Note that sequences satisfying \eqref{43} and \eqref{44} are always doubly infinite, unless $q=M+1$.  Moreover, if $(c_i) \in \vv_q'$, then by Proposition \ref{p27} (ii), the inequalities \eqref{43} and \eqref{44} hold whenever $c_1\cdots c_n\ne M^n$ and $c_1\cdots c_n\ne 0^n$, respectively.

\begin{lemma}\label{l44}
If $1<p<q \le M+1$, then
\begin{equation*}
\uu'_p \subseteq\uu'_q,\quad
\vv'_p \subseteq\vv'_q\qtq{and}
\vv'_p \subseteq\uu'_q.
\end{equation*}
\end{lemma}

\begin{proof}
The inclusions $\uu'_p \subseteq\uu'_q$ and $\vv'_p \subseteq\vv'_q$ follow from Lemma \ref{l42}, Definition \ref{d43} and Proposition \ref{p26} (ii). Moreover, since the map $q \mapsto \alpha(q)$ is \emph{strictly}  increasing, the inclusion $\vv'_p \subseteq\uu'_q$ holds as well.
\end{proof}

The sets $\uu_p'$ and $\vv_p'$ only have a nontrivial structure if $p \in (\tilde q,M+1)$ and $p \in [\tilde q,M+1)$ respectively. Here $\tilde q$ is the generalized Golden ratio whose exact value (depending on $M$) is given in Theorem \ref{t33} (ii). We will illustrate this in the following examples. The results of Examples \ref{e45} (ii), (iii) and (iv) are already established by Baker (\cite{[B]}). We provide short alternative proofs of these results by refining an idea contained in the proof of Theorem 3 in [EJK]. 

\begin{examples}\label{e45}\mbox{}

\begin{enumerate}[\upshape (i)]
\item First, let $M=2m$ and $q=m+1=\tilde q$ for some $m\ge 1$, so that $\alpha(q)=m^{\infty}$. 
It follows from Definition  \ref{d43} that 
\begin{equation*}
\vv_{\tilde q}'=\set{0^{\infty},M^{\infty}} \cup \set{km^{\infty} \ : 0 < k < M} 
\cup\set{0^n l m^{\infty}, M^n\overline{l}m^{\infty}\ :\ n\ge 1, \ 1 \le l \le m}.  
\end{equation*}
If  $M=2m-1$ and $q=\tilde q$ for some $m\ge 1$, then $\alpha(q)=(mm^-)^{\infty}$. Definition \ref{d43} yields in this case that  
\begin{align*}
\vv_{\tilde q}' &=  \set{0^{\infty},M^{\infty}}
\cup \set{k(mm^-)^{\infty}, k(m^-m)^{\infty} \ : 0 <  k < M} \\
 & \cup \set{0^n l(mm^-)^{\infty},
0^n r(m^-m)^{\infty} \ : n \ge 1,  1 \le  l <  m, \ 1 \le r \le m} \\
 & \cup 
\set{M^n l (mm^-)^{\infty},
M^n r (m^-m)^{\infty} 
\ :\ n\ge 1, m-1 \le  l < M, \ m \le r < M}.
\end{align*}
In both cases (i.e., for all $M \ge 1$), we infer from Lemma \ref{l42} that 
\begin{equation*} 
\uu_{\tilde q}'=\set{0^{\infty},M^{\infty}}.
\end{equation*}
Hence, by Lemma ~\ref{l44}, $\uu_q=\vv_q=\set{0, M/(q-1)}$ if $q \in (1,\tilde q)$ and $\vv_q \supseteq \uu_q \supsetneq \set{0,M/(q-1)}$ if $q \in (\tilde q, M+1]$.
\item  In the remaining Parts (ii), (iii) and (iv) of this example, $M \ge 1$ and  $q \in (1, \tilde q)$ are fixed but arbitrary, and $m$ is defined as before: $m=(M+1)/2$ if $M$ is odd, and $m=M/2$ if $M$ is even. We deduced in (i) that $\uu_q=\set{0,M/(q-1)}$. We now strengthen this result by proving that \emph{each} $x$ in the interior of $J_q$ has $2^{\aleph_0}$ expansions. Let $x$ be a number in the interior of $J_q$. From Theorem \ref{t33} (ii) and (vi) we infer that 
\begin{equation*}
1 = \frac{m-1}{\tilde q} + \sum_{j=2}^{\infty} \frac{M}{\tilde q ^j}.
\end{equation*}
Hence there exists a positive integer $k \ge 2$ such that 
\begin{equation}\label{unc1}
1 <  \frac{m-1}{q} + \frac{M}{q^2}+\cdots+ \frac{M}{q^k}. 
\end{equation}
Let $n_1,n_2,n_3,\ldots$ be the strictly increasing sequence consisting of positive integers which are \emph{not} multiples of $k$, i.e., $\set{n_r:r \ge 1}=\NN \setminus k \NN$. 
Furthermore, choose $k$ large enough so that the inequalities 
\begin{equation}\label{unc2}
\sum_{j=1}^{\infty} \frac{m}{q^{kj}}  \le x \quad \text{and} \quad  x < \sum_{j=1}^{\infty} \frac{M}{q^{n_j}}
\end{equation}
hold as well. 
Let $(\delta_i)=\delta_1\delta_2 \cdots$ be an arbitrary sequence in $\set{0,m}^{\mathbb{N}}$, and let 
\begin{equation*}
x':=x- \sum_{j=1}^{\infty} \frac{\delta_j}{q^{kj}}.
\end{equation*}
Note that $x' \ge 0$ by the first inequality of \eqref{unc2}. We now define recursively an expansion $(\varepsilon_i)=\varepsilon_1 \varepsilon_2 \cdots$ of $x'$ by applying a variant of the greedy algorithm as follows.  If for some $n \ge 1$, the digits $\varepsilon_1, \varepsilon_2,\ldots,\varepsilon_{n-1}$ are already defined (no condition if $n=1$) and $n$ is not a multiple of $k$, then let $\varepsilon_{n}$ be the largest digit in the (whole) alphabet  $A$ such that 
\begin{equation}\label{unc3}
\left(\sum_{j=1}^{n-1} \frac{\varepsilon_j}{q^j} \right)+ \frac{\varepsilon_{n}}{q^{n}} \le x'. 
\end{equation}
If $n$ is a multiple of $k$, then let $\varepsilon_n$ be the largest element in $\set{0,\ldots,m-1}$ for which \eqref{unc3} holds. 
By the second inequality of \eqref{unc2}, there exists an index $\ell$ which is not a multiple of $k$ such that $\varepsilon_{\ell} < M$. By \eqref{unc1}, not a single block $\varepsilon_{r+1} \cdots \varepsilon_{r+k}$ of length $k$ in the sequence $\varepsilon_{\ell+1}\varepsilon_{\ell+2} \cdots$ can be \emph{maximal}, that is, it cannot be the case that both $\varepsilon_{r+i}=m-1$ for $i$ such that $r+i$ is a multiple of $k$ and 
$\varepsilon_{r+i}= M$ for the remaining indices $r+i$ $(1 \le i \le k)$. Hence there are infinitely many indices $s$ such that 
\begin{equation*}
x'-\frac{1}{q^{s}} < \sum_{j=1}^{s} \frac{\varepsilon_j}{q^j}  \le x'.
\end{equation*}
Letting $s \to \infty$ along these indices, we see that $(\varepsilon_i)$ is indeed an expansion of $x'$. The expansion $(c_i)$ with $c_i=\varepsilon_i$ if $i$ is not a multiple of $k$ and with $c_{kj}=\delta_j+\varepsilon_{kj}$ ($j \ge 1$) is an expansion of $x$. It remains to observe that distinct sequences $(\delta_i) \in \set{0,m}^{\infty}$ give rise to distinct expansions $(c_i)$ of $x$ because $\delta_j=0$ if and only if $c_{kj} < m$ $(j \ge 1)$.  

\item Let $\ee_{q}(x)$ be the set of all possible expansions of $x \in J_q$, and let $\ee_{q}^n(x)$ be the set of all possible prefixes of length $n$ $(n \ge 1)$ of sequences belonging to $\ee_{q}(x)$: 
\begin{equation*}
 \ee_{q}^{n}(x):=\set{(c_1 ,\ldots,c_n) \in A^n: \exists (c_{n+1},c_{n+2}, \ldots) \in A^{\NN} \text{ so that } (c_i) \in \ee_{q}(x)}. 
\end{equation*}
Finally, let $\nn_n(x,q):= \vert \ee_{q}^n(x) \vert$. The analysis of (ii) can be used to show that $\nn_n(x,q)$ grows exponentially fast as a function of $n$ for $x$ in the interior of $J_q$, in the rather strong sense that there exists a constant $c=c(q,M)>0$ that does not depend on $x$ so that for \emph{each} $x$ in the interior of $J_q$, the inequality 
\begin{equation*}
\liminf_{n \to \infty} \frac{\log_{M+1} \left(\nn_n(x,q) \right)}{n} \ge c  
\end{equation*}
holds. Indeed, fix $x$ in the interior of $J_q$, and let $(c_i)$ and $(d_i)$ be two distinct expansions of $x$. Thanks to (ii), we can assume that $(c_i)$ and $(d_i)$ do not end with $M^{\infty}$.  If $r=r(x)$ is the smallest index such that $c_r \not=d_r$, then $c_{r}c_{r+1}c_{r+2}\cdots$ and $d_{r}d_{r+1}d_{r+2}\cdots$ must be expansions of a number belonging to 
\begin{equation*}
\left[L,R \right):=\left[\frac{1}{q}, \frac{M-1}{q} + \frac{M}{q\left(q-1\right)} \right).
\end{equation*} 
Hence, it is sufficient to show that the assertion holds for each $x$ in $\left[L, R \right)$. It follows from the proof of (ii) that one may take $c=\log 2 / (k \cdot \log (M+1))$, where $k$ is the smallest positive integer such that \eqref{unc1} and the two inequalities 
\begin{equation}\label{unc4}
\sum_{j=1}^{\infty} \frac{m}{q^{kj}}  \le \frac{1}{q} \quad \text{and} \quad  \frac{M-1}{q} + \frac{M}{q \left(q-1\right)} \le \sum_{j=1}^{\infty} \frac{M}{q^{n_j}}
\end{equation}
hold, where, as before, $n_1,n_2,n_3,\ldots$ is the strictly increasing sequence of all positive integers which are not multiples of $k$.  

\item The Tychonov product topology on $A^{\NN}$ is induced by the metric $d$ defined as follows: 
\begin{equation*}
d((c_i),(d_i)):=
\begin{cases}
(M+1)^{-n} & \text{ if } (c_i) \not = (d_i) \text{ and }  n \text{ is the first index such that } c_n \not = d_n,  \\
0  & \text{ if } (c_i)=(d_i).
\end{cases}
\end{equation*}
We want to show that in the metric space $\left (A^{\NN}, d \right)$, the Hausdorff dimension of the subset $\ee_{q}(x)$ is, for \emph{each} $x$ in the interior of $J_q$, bounded below by the same constant $c(q,M)$ that we found in (iii). We may (and will) again assume that $x$ belongs to the interval $\left[L,R\right)$ that we defined in (iii).

Let $k$ be a positive integer satisfying \eqref{unc1} and both inequalities in \eqref{unc4}, and define the Lipschitz map $f: A^{\NN} \to A^{\NN}$ by 

\begin{equation*}
(f(c_i))_n: = 
\begin{cases}
0 & \text{ if } n \text{ is not a multiple of } k,\\
0 & \text{ if } n \text{ is a multiple of } k \text{ and } c_n < m,\\
1 & \text{ if } n \text{ is a multiple of } k \text{ and } c_n \ge m.
\end{cases}
\end{equation*}
We have seen in Part (ii) that $f\left(\ee_q(x)\right)= f \left(A^{\NN}\right)$.  The bi-Lipschitz map $g: f \left(A^{\NN}\right) \to [0,1]$ given by $g(f(c_i))=\sum_{n=1}^{\infty} (f(c_i))_n \cdot (M+1)^{-n}$ maps $f \left(A^{\NN}\right)$ onto the attractor of the iterated function system that consists of the similarities $T:[0,1] \to [0,1]$ and $S:[0,1] \to [0,1]$, defined by $T(y)=y \cdot (M+1)^{-k}$ and $S(y)=(M+1)^{-k} + y \cdot (M+1)^{-k}$, $0 \le y \le 1$. By Propositions 9.6 and 9.7 in \cite{[F]}, the Hausdorff dimension of this attractor equals the solution of the equation $2 \cdot (M+1)^{-ks} =1$. 
Hence 
\begin{equation*}
\dim_H \left(\ee_q(x)\right) \ge \dim_H f \left(\ee_q (x) \right) = \dim_H  f \left(A^{\NN}\right) = \dim_H g \left( f \left(A^{\NN} \right)\right) =\frac{\log 2}{k \cdot \log (M+1)}.
\end{equation*}     
\end{enumerate}
\end{examples}

We infer from Definition \ref{d43} the following useful characterizations of $\vv_q$ which were already obtained in \cite{[DVK2]} in case $q \in (M,M+1]$ :

\begin{lemma}\label{l47}
Let $(x,q) \in \bfj$, and write $(\alpha_i)=\alpha(q)$, $(\beta_i)=\beta(q)$ and $(a_i)=a(x,q)$.
The following conditions are equivalent:
\begin{align}
&x\in\vv_q;\label{45}\\
&\overline{a_{n+1}a_{n+2}\cdots }\le\alpha(q)\qtq{whenever}a_n>0;\label{46}\\
&\overline{a_{n+1}a_{n+2}\cdots }\le\beta(q)\qtq{whenever}a_n>0.\label{47}
\end{align}
\end{lemma}

\begin{proof}
Assume \eqref{45}.
If $(c_i)\in\vv'_q$ and $x=\pi_q((c_i))$, then $(c_i)=a(x,q)$ by \eqref{43} and Proposition \ref{p27} (ii), and then \eqref{46} follows from \eqref{44}.

Conversely, if $a(x,q)$ satisfies \eqref{46}, then $(c_i):=a(x,q)$ satisfies \eqref{44} by \eqref{46}, and  \eqref{43} by Proposition \ref{p27} (ii).
Hence \eqref{45} is satisfied because $a(x,q)$ is always infinite.

Since $\alpha(q)\le\beta(q)$, \eqref{46} implies \eqref{47}.
In order to show the converse implication, it suffices to show that if there exists a positive integer $n$ such that
\begin{equation*}
a_n>0\qtq{and}\overline{a_{n+1}a_{n+2}\cdots }>\alpha_1\alpha_2\cdots ,
\end{equation*}
then there exists also a positive integer $m$ such that
\begin{equation*}
a_m>0\qtq{and}\overline{a_{m+1}a_{m+2}\cdots }>\beta_1\beta_2\cdots .
\end{equation*}
If the greedy expansion $(\beta_i)$ is infinite, then $(\beta_i)=(\alpha_i)$ and we may choose $m=n$. 
If $(\beta_i)$ has a last nonzero digit $\beta_k$, then
$(\alpha_i)=(\alpha_1\cdots \alpha_k)^{\infty}$ with $\alpha_1\cdots \alpha_k=(\beta_1\cdots \beta_k)^-$, and
thus $\alpha_k<M$. 
Since we have
\begin{equation*}
\overline{a_{n+1}a_{n+2}\cdots }>(\alpha_1\cdots \alpha_k)^{\infty}
\end{equation*}
by assumption, there exists a nonnegative integer $j$ satisfying
\begin{equation*}
\overline{a_{n+1}\cdots a_{n+jk}}=(\alpha_1\cdots \alpha_k)^j
\qtq{and}
\overline{a_{n+jk+1}\cdots a_{n+(j+1)k}}>\alpha_1\cdots \alpha_k.
\end{equation*}
Putting $m:=n+jk$ it follows that
\begin{equation*}
a_{m}>0\qtq{and}\overline{a_{m+1}\cdots a_{m+k}}\ge\beta_1\cdots \beta_k.
\end{equation*}
It follows from our assumption $\overline{a_{n+1}a_{n+2}\cdots }>\alpha_1\alpha_2\cdots$ that
$(\alpha_i)<M^{\infty}$ and $(a_i) < M^{\infty}$. 
By Proposition~\ref{p22}, $(a_i)$ has no tail equal to $M^{\infty}$, so that $\overline{a_{m+k+1}a_{m+k+2}\cdots }>0^{\infty}$. 
We conclude that
\begin{equation*}
\overline{a_{m+1}a_{m+2}\cdots}>\beta_1\beta_2\cdots .\qedhere
\end{equation*}
\end{proof}

\begin{lemma}\label{l48}
Let $q \in (1,M+1]$.
\begin{enumerate}[\upshape (i)]
\item We have $\uu_q \subseteq \vv_q$.
\item The sets $\uu_q$ and $\vv_q$ are \emph{symmetric} in $J_q$, i.e., $\ell(\uu_q)=\uu_q$ and $\ell(\vv_q)=\vv_q$.
\item The set $\vv_q$ is closed.
\end{enumerate}
\end{lemma}

\begin{proof}
(i) If $x\in\uu_q$, then its unique expansion $(c_i)$ is equal to $a(x,q)$, and \eqref{46} follows from \eqref{42}.
\medskip

(ii) The set $\uu_q$ is symmetric because $(c_i)$ is an expansion of $x$ if and only if $(M-c_i)$ is an expansion of $\ell(x)$.

If $q\in(1,M+1)$ and $x \in J_q$, then $a(x,q)$ is doubly infinite by Proposition~\ref{p22} (ii), and hence the sequence $\overline{a(x,q)}$ is also infinite. 
Using \eqref{46} and  applying Proposition~\ref{p27} (ii) hence we infer that $a(\ell(x),q)=\overline{a(x,q)}$, and then $\ell(x)\in\vv_q$ by \eqref{26} and \eqref{46}. 
If $q=M+1$, then $\vv_q = J_q = [0,1]$ is symmetric.
\medskip 

(iii) We prove that $\vv_q$ is closed by showing that its complement in $J_q$ is
open. If $(a_i)$ is the quasi-greedy expansion of some $x \in J_q
\setminus \vv_q$, then there exists an integer $n >0$ such that
\begin{equation*}
a_n >0 \quad \mbox{and} \quad \overline{a_{n+1} a_{n+2} \cdots}
> \alpha_1 \alpha_2 \cdots .
\end{equation*}
Let $m$ be such that
\begin{equation}\label{48}
\overline{a_{n+1} \cdots a_{n+m}} > \alpha_1 \cdots \alpha_m,
\end{equation}
and let
\begin{equation*}
y=\sum_{i=1}^{n+m} \frac{a_i}{q^i}.
\end{equation*}
According to Lemma~\ref{l29} the greedy expansion of $y$ is given by
$a_{1} \cdots a_{n+m}0^{\infty}$. Therefore the quasi-greedy expansion of each
number $v \in (y,x]$ starts with the block $a_1 \cdots a_{n+m}$.
It follows from \eqref{48} that
\begin{equation*}
(y,x] \cap \vv_q = \varnothing.
\end{equation*}
Since $x \in J_q \setminus \vv_q$ is arbitrary and $\vv_q$ is symmetric, there also exists a number $z > x$ such that
\begin{equation*}
[x, z) \cap \vv_q = \varnothing.
\qedhere
\end{equation*}
\end{proof}

\begin{lemma}\label{l49}
Let $q\in(1,M+1]$, $(\alpha_i)=\alpha(q)$, and let $(b_i)$ be the greedy expansion of some $x \in J_q$.
Suppose that for some $n \geq 1$,
\begin{equation*}
b_n > 0 \quad \mbox{and} \quad \overline{b_{n+1}b_{n+2}\cdots} > \alpha_1\alpha_2 \cdots .
\end{equation*}
Then:
\begin{enumerate}[\upshape (i)]
\item there exists a number $z > x$ such that $[x,z] \cap \uu_q =\varnothing$ and $(x,z] \cap \vv_q=\varnothing$; 
\item if $b_j >0$ for some $j >n$, then there exists a number
$y < x$ such
that $[y,x] \cap \uu_q =\varnothing$.
\end{enumerate}
\end{lemma}

\begin{proof}
(i) Choose a positive integer $M >n$ such that
\begin{equation*}
\overline{b_{n+1}\cdots b_M} > \alpha_1 \cdots \alpha_{M-n}.
\end{equation*}
Applying Lemma~\ref{l210} choose a greedy sequence $(c_i) > (b_i)$
such that $c_1\cdots c_M$ $=b_1 \cdots b_M$. Then $(c_i)$ is the greedy
expansion of some $z > x$.
If $(d_i)$ is the greedy expansion of an element in $[x,z]$ or the quasi-greedy expansion of an element in $(x,z]$, then
$(d_i)$ also begins with  $b_1\cdots b_M$ and hence
\begin{equation*}
d_n>0
\qtq{and}
\overline{d_{n+1} \cdots d_{M}}> \alpha_1 \cdots \alpha_{M-n} .
\end{equation*}
The assertion follows from Lemmas~\ref{l42} and \ref{l47}. 

(ii) Suppose that $b_j >0$ for some $j > n$. It follows from Lemma~\ref{l29} that
$(c_i):=b_1 \cdots b_n 0^{\infty}$ is the greedy expansion of some $y < x$.
If $(d_i)$ is the greedy expansion
of some element in $[y,x]$, then $(c_i) \leq (d_i) \leq (b_i)$ and
$d_1 \cdots d_n=b_1 \cdots b_n$.
Therefore
\begin{equation*}
d_n > 0 \qtq{and} \overline{d_{n+1}d_{n+2}\cdots} \geq \overline{b_{n+1}b_{n+2}\cdots}
> \alpha_1\alpha_2 \cdots. 
\end{equation*}
Invoking  Lemma \ref{l42} again, we conclude that 
$[y,x] \cap  \uu_q = \varnothing$.
\end{proof}

\section{Two  lemmas on density}\label{s5}

The two results of this section are crucial for the proof of our main theorems.
Their proofs are based on the construction of special convergent sequences, interesting in themselves.

Let $q\in(1,M+1]$ and $(\alpha_i)=\alpha(q)$.
We recall that 
\begin{equation}\label{51}
\alpha_{n+1}\alpha_{n+2}\cdots\leq\alpha_1\alpha_2\cdots\qtq{for all}n\ge 0.
\end{equation}
Furthermore, if $q\in\uuu$, then
\begin{equation}\label{52}
\overline{\alpha_{n+1}\alpha_{n+2}\cdots} <\alpha_1\alpha_2\cdots\qtq{for all}n\ge 0;
\end{equation}
moreover, by Theorem~\ref{t32} (v) there exist arbitrarily large integers $m$ such that
\begin{equation}\label{53}
\overline{\alpha_{k+1} \cdots \alpha_{m}} < \alpha_1 \cdots \alpha_{m-k},\quad k=0,\ldots, m-1.
\end{equation}

We recall from Section \ref{s1} that $A_q$ denotes the set of elements $\vv_q \setminus \uu_q$ having a finite greedy expansion.

\begin{lemma}\label{l51}
If $q \in \vv$, then $\abs{A_q}=\aleph_0$.
Furthermore, for each $x \in \uu_q$ there exists a sequence $(x_i)$ of elements in $A_q$ such that  $x_i \to  x$ and $a(x_i,q) \to a(x,q)$. Morever, one may choose the sequence $(x_i)$ to be increasing if $x  \in \uu_q \setminus \set{0}$.  
\end{lemma}

\begin{proof}
Let $x \in \uu_q \setminus\set{0}$, and denote by $(c_i)$ the unique expansion of $x$.
Since $\overline{c_1 c_2 \cdots}
\ne {M^{\infty}}$ is
greedy, we infer from Lemma~\ref{l210} (i)
that there exists a sequence $1 \leq n_1 < n_2 < \cdots$, such that
for each $i \geq 1$,
\begin{equation}\label{54}
c_{n_i} > 0 \quad \mbox{and} \quad
\overline{c_{m+1} \cdots c_{n_i}} < \alpha_1
\cdots \alpha_{n_i - m} \quad \mbox{if }  m < n_i \, \,
\mbox{and} \, \, c_m >0.
\end{equation}
Now consider for each $i \geq 1$ the sequence $(b_j^i)$, given by
\begin{equation*}
(b_j^i) = c_1 \cdots c_{n_i} 0 ^{\infty},
\end{equation*}
and define the number $x_i$ by
\begin{equation*}
x_i = \sum_{j=1}^{\infty} \frac{b_j^i}{q^j}.
\end{equation*}
Since $(c_i)$ has infinitely many nonzero digits, we have $x_i\uparrow x$.
According to Lemma~\ref{l29} the sequence $(b_j^i)$ is the finite
greedy expansion of the number $x_i$, $i \geq 1$. Moreover, the increasing sequence
$(x_i)_{i \geq 1}$ converges to $x$ as $i$ goes to infinity. We claim that
$x_i \in A_q$ for each
$i \geq 1$. Note that $x_i \notin \uu_q$ because
the quasi-greedy sequence $(a_j^i)$, given by
\begin{equation*}
(c_1 \cdots c_{n_i})^{-} \alpha_1 \alpha_2 \cdots,
\end{equation*}
is another expansion of $x_i$. By Lemma~\ref{l47} it remains to prove that
\begin{equation}\label{55}
a_j^i >0    \Longrightarrow \overline{a_{j+1}^i a_{j+2}^i \cdots}
 \leq \alpha_1 \alpha_2 \cdots.
\end{equation}
If $j < n_i$ and $a_j^i>0$, then
\begin{equation*}
\overline{a_{j+1}^i \cdots a_{n_i}^i}
= \overline{(c_{j+1} \cdots c_{n_i})^-} \leq \alpha_1 \cdots \alpha_{n_i - j}
\end{equation*}
by \eqref{54}, and
\begin{equation*}
\overline{a_{n_i+1}^i a_{n_i+2}^i \cdots}=
\overline{\alpha_1 \alpha_2 \cdots} \le
\alpha_{n_i-j+1} \alpha_{n_i-j+2} \cdots
\end{equation*}
by Proposition~\ref{p31}. 
If $j \ge n_i$, then \eqref{55}
follows directly from Proposition~\ref{p31}.
Finally, note that $q^{-n} \to 0$ as $n \to \infty$, and that $q^{-n} \in A_q$ for each $n \in \NN$ because $\alpha(q^{-n})=0^n\alpha(q)$. 
Since there are only countably many finite expansions, we conclude that $\abs{A_q}=\aleph_0$.
\end{proof}

\begin{lemma}\label{l52}
If $q \in \uuu$, then for each $x \in A_q$ there exists a sequence $(x_i)$ of elements in $\uu_q$  such that $x_i\uparrow x$.
\end{lemma}

\begin{proof}
Let $x \in A_q$.
If $b_n$ is the last nonzero element of
$(b_i)=b(x,q)$, then
\begin{equation*}
a(x,q)=(a_i)=(b_1\cdots b_n)^{-} \alpha_1 \alpha_2 \cdots .
\end{equation*}
Since $q \in \uuu$ by assumption, there exists a sequence
$1 \leq m_1 < m_2 < \cdots $ such that \eqref{53}
is satisfied with $m=m_i$ for all $i \geq 1$.
We may assume that $m_1 \ge n$ and $m_{i+1} \ge 2m_i$ for each $i \ge 1$. 
Consider for each $i \geq 1$ the sequence $(b_j^i)$, given by
\begin{equation*}
(b_j^i)=(b_1 \cdots b_n)^-(\alpha_1 \cdots \alpha_{m_i}\overline
{\alpha_1 \cdots \alpha_{m_i}})^{\infty},
\end{equation*}
and define the number  $x_i$ by
\begin{equation*}
x_i= \sum_{j=1}^{\infty}\frac{b_j^i}{q^j}.
\end{equation*}
Note that the sequence $(x_i)$ converges to $x$ as $i$ goes to infinity. 
Next we show that $x_i \in
\uu_q$ for all $i \geq 1$.
According to Lemma~\ref{l42} it suffices to verify that
\begin{equation}\label{56}
b_{m+1}^i b_{m+2}^i \cdots < \alpha_1 \alpha_2 \cdots \quad \mbox{whenever}
\quad
b_m^i <  M
\end{equation}
and
\begin{equation}\label{57}
\overline{b_{m+1}^i b_{m+2}^i \cdots} <
\alpha_1 \alpha_2 \cdots \quad \mbox{whenever} \quad b_m^i > 0.
\end{equation}
According to \eqref{52} we have
\begin{equation*}
\overline{\alpha_{m_i+1} \cdots \alpha_{2m_i}} \leq
\alpha_1 \cdots \alpha_{m_i}.
\end{equation*}
Note that this inequality cannot be an equality, for otherwise it would
follow from Theorem \ref{t33} (v) that
\begin{equation*}
(\alpha_i)=(\alpha_1 \cdots \alpha_{m_i} \overline
{\alpha_1 \cdots \alpha_{m_i}})^{\infty}.
\end{equation*}
However, this sequence does not satisfy \eqref{52} for $k=m_i$.
Therefore
\begin{equation*}
\overline{\alpha_{m_i+1} \cdots \alpha_{2m_i}} <
\alpha_1 \cdots \alpha_{m_i}
\end{equation*}
or equivalently
\begin{equation}\label{58}
\overline{\alpha_1 \cdots \alpha_{m_i}} < \alpha_{m_i+1} \cdots \alpha_{2m_i}.
\end{equation}
If $m \geq n$, then \eqref{56} and \eqref{57}
follow from \eqref{51}, \eqref{53} and \eqref{58}. Now we verify
\eqref{56} and \eqref{57} for $m < n$. Fix $m < n$. If $b_m^i <  M$, then
\begin{equation*}
b_{m+1}^i \cdots b_n^i = (b_{m+1} \cdots b_n)^-
< b_{m+1} \cdots b_n
\leq \alpha_1 \cdots \alpha_{n-m},
\end{equation*}
where the last inequality follows from the fact that $(b_i)$ is a
greedy expansion and $b_m=b_m^i <  M.$ Hence
\begin{equation*}
b_{m+1}^i b_{m+2}^i \cdots < \alpha_1 \alpha_2 \cdots .
\end{equation*}
Suppose that $b_m^i = a_m >0.$ Since
\begin{equation*}
\overline{a_{m+1} a_{m+2} \cdots} \leq \alpha_1 \alpha_2 \cdots
\end{equation*}
by Lemma~\ref{l47}, and $b_{m+1}^i \cdots b_n^i = a_{m+1} \cdots a_n$, it suffices
to verify that
\begin{equation*}
\overline{b_{n+1}^i b_{n+2}^i \cdots} < \alpha_{n-m+1} \alpha_{n-m+2} \cdots .
\end{equation*}
This is equivalent to
\begin{equation} \label{59}
\overline{\alpha_{n-m+1} \alpha_{n-m+2} \cdots} <
 (\alpha_1 \cdots \alpha_{m_i}\overline{\alpha_1 \cdots \alpha_{m_i}})^{\infty}.
\end{equation}
Since $n \le  m_i$ for all $i \geq 1$, we infer from \eqref{53} that
\begin{equation*}
\overline{\alpha_{n-m+1} \cdots \alpha_{m_i}} < \alpha_1 \cdots
\alpha_{m_i-(n-m)},
\end{equation*}
and \eqref{59} follows. It follows from \eqref{58} and the assumption $m_{i+1} \ge 2 m_i$ $(i \ge 1)$ that the sequence $(x_i)$ is strictly increasing. 
\end{proof}

\section{Proof of Theorem~\ref{t11} and Proposition \ref{p12}}\label{s6}

First we prove the relation \eqref{11} between $A_q$ and $B_q$ as stated in Proposition  \ref{p12} (ii):

\begin{lemma}\label{l61}
If $q \in \uuu\setminus\set{M+1}$, then
\begin{equation*}
B_q = \ell(A_q),
\end{equation*}
and the  greedy expansion of each $x \in B_q$ ends with $\overline{\alpha(q)}$.
\end{lemma}

\begin{proof}
First suppose that $x \in A_q$ has a
finite greedy expansion $(b_i)$ with a last nonzero element $b_n$.
Since $q \ne M+1$, the sequence 
\begin{equation*}
(c_i):= \overline{(b_1 \cdots b_n)^- \alpha_1 \alpha_2 \cdots}
\end{equation*}
is an infinite expansion of $\ell(x)$. 
In order to show that $(c_i)$ is the greedy expansion of $\ell(x)$, we verify that $(c_i)$ satisfies the inequalities \eqref{25} of Proposition \ref{p27}.
By Proposition \ref{p31} it is suficient to show that if $1\le k<n$ and $b_k>0$, then 
\begin{equation*}
\overline{(b_{k+1} \cdots b_n)^- \alpha_1 \alpha_2 \cdots}<\alpha_1 \alpha_2 \cdots.
\end{equation*}
Since $x\in\vv_q$ and $a(x,q)=(b_1\cdots b_n)^- \alpha_1 \alpha_2 \cdots$ by Lemma \ref{l28} (ii), we have (see \eqref{46}) 
\begin{equation*}
\overline{(b_{k+1} \cdots b_n)^- \alpha_1 \alpha_2 \cdots}\le \alpha_1 \alpha_2 \cdots.
\end{equation*}
We cannot have an equality here because it would imply $\overline{\alpha_{n-k+1}\alpha_{n-k+2}\cdots}=\alpha_1 \alpha_2 \cdots$, contradicting $q\in\uuu$ (see Proposition \ref{p31}  again). It follows from the symmetry of $\uu_q$ and $\vv_q$  (see Lemma ~\ref{l48}) that $\ell(x)\in B_q$.

Conversely, suppose that $x \in B_q$ and let $(b_i)$ be its infinite greedy expansion.
Since $x\notin\uu_q$, there exists
 a \emph{smallest} positive integer $n$ for which 
\begin{equation}\label{61}
b_n >0 \qtq{and}\overline{b_{n+1}b_{n+2} \cdots}
\geq \alpha_1 \alpha_2 \cdots .
\end{equation}
By Lemma~\ref{l47}, we have necessarily 
$b_{n+1}b_{n+2}\cdots= \overline{\alpha_1\alpha_2\cdots}$. 
We claim that
\begin{equation*}
(c_i)= \overline{(b_1 \cdots b_n)^-} 0 ^{\infty}
\end{equation*}
is the greedy expansion of $\ell(x)\in A_q$.
We have to show that if $1\le k<n$ and $b_k>0$, then 
\begin{equation*}
\overline{(b_{k+1}\cdots b_n)^-}0^{\infty}
<\alpha_1\alpha_2\cdots .
\end{equation*}
Since $(\alpha_i)$ has infinitely many nonzero digits, it suffices to show that 
\begin{equation*}
\overline{b_{k+1} \cdots b_n}<\alpha_1\cdots\alpha_{n-k}.
\end{equation*}
By the minimality of $n$ we have 
\begin{equation*}
\overline{b_{k+1}b_{k+2} \cdots }
<\alpha_1\alpha_2\cdots .
\end{equation*}
Hence 
\begin{equation*}
\overline{b_{k+1} \cdots b_n}\le\alpha_1\cdots\alpha_{n-k},
\end{equation*}
and it remains to exclude the equality.
However, in case of equality we would obtain the impossible relations
\begin{equation*}
\alpha_1\alpha_2\cdots
>\overline{b_{k+1}b_{k+2} \cdots }
=\alpha_1\cdots\alpha_{n-k}\alpha_1\alpha_2\cdots
\ge \alpha_1\cdots\alpha_{n-k}\alpha_{n-k+1}\alpha_{n-k+2}\cdots. \qedhere
\end{equation*}
\end{proof}

Now we are ready for the proof of Theorem~\ref{t11}.
The following proof will also establish the stronger properties stated in Proposition \ref{p12}.

\begin{proof}[Proof of Theorem~\ref{t11}]
(i) and (ii) Suppose that $q\in\uuu\setminus\set{M+1}$. Since $\uu_q \subseteq\vv_q$ and $\vv_q$ is closed by Lemma \ref{l48}, we have $\overline{\uu_q}  \subseteq \vv_q$. 
Conversely, it follows from Lemmas \ref{l52}, \ref{l61} and the symmetry of $\uu_q$ and $\vv_q$ that  $\vv_q \setminus \uu_q \subseteq\overline{\uu_q}$, and this implies $\vv_q \subseteq\overline{\uu_q}$. We infer from Lemmas \ref{l51}, \ref{l61} that both $A_q$ and $B_q$ are dense in $\overline{\uu_q}=\vv_q$, and $\abs{A_q}=\abs{B_q}=\aleph_0$.
Therefore $\abs{\vv_q\setminus\uu_q}=\abs{A_q\cup B_q}=\aleph_0$.
If $q = M+1$, the properties stated in Theorem \ref{t11} and Proposition \ref{p12} are well known.
\medskip

(iii) and (iv) Fix $q \in \uuu$ and let $(b_i)$ be the greedy expansion of a number $x \in \vv_q\setminus \uu_q$. 
Let $n$ be the \emph{smallest} positive integer for which \eqref{61} holds and let $(d_i)$ be another expansion of $x$. 
Then $(d_i) < (b_i)$, and hence there exists a \emph{smallest} integer $j\ge1$ for which $d_j < b_j$. 
First we show that $j \geq n$. 
Assume on the contrary that $j < n$. 
Then $b_j >0$, and by minimality of $n$ we have
\begin{equation*}
b_{j+1}b_{j+2} \cdots > \overline{\alpha_1 \alpha_2 \cdots}.
\end{equation*}
From Proposition~\ref{p31} (ii) we know that $\overline{\alpha_1 \alpha_2 \cdots}$ is the greedy expansion of
$\ell(x)$, and thus
\begin{equation*}
\sum_{i=1}^{\infty} \frac{d_{j+i}}{q^i} 
=b_j - d_j + \sum_{i=1}^{\infty} \frac{b_{j+i}}{q^i} 
>1+\sum_{i=1}^{\infty} \frac{\overline{\alpha_i}}{q^i}
=1 + \frac{M}{q-1} - 1 
=  \frac{M}{q-1},
\end{equation*}
which is impossible.
If $j=n$, then $d_n=b_n^-$, for otherwise we have 
\begin{equation*}
2 \leq \sum_{i=1}^{\infty} \frac{d_{n+i}}{q^i} \leq \frac{M}{q-1},
\end{equation*}
which is also impossible. 
Indeed, the last condition would imply $q\le (M+2)/2$, while $q\in\uuu$ implies $q>\tilde q\ge (M+2)/2$ by  Theorem \ref{t33} (ii).
Now we distinguish between two cases.

If $j=n$ and
\begin{equation}\label{62}
\overline{b_{n+1} b_{n+2} \cdots} > \alpha_1 \alpha_2 \cdots,
\end{equation}
then by Lemma~\ref{l49} and (i) we have $b_{r}=0$ for all $r >n$, from which
it follows that $(d_{n+i})$ is an expansion of $1$. 
Hence, if $q \in
\uu$ and \eqref{62} holds, then the only expansion of $x$ starting with $(b_1 \cdots b_n)^-$ is given by $(c_i):=( b_1 \cdots b_n)^- \alpha_1\alpha_2 \cdots$. 
If $q \in \uuu \setminus \uu$ and
\eqref{62} holds, then any expansion $(c_i)$ starting with $(b_1 \cdots
b_n)^-$ is an expansion of $x$ if and only if $(c_{n+i})$ is one of the
expansions listed in Theorem  \ref{t32} (vi).

If $j=n$ and
\begin{equation}\label{63}
\overline{b_{n+1} b_{n+2} \cdots} = \alpha_1 \alpha_2 \cdots,
\end{equation}
then
\begin{equation*}
\sum_{i=1}^{\infty} \frac{d_{n+i}}{q^i}=
1+ \sum_{i=1}^{\infty} \frac{b_{n+i}}{q^i}  =
\sum_{i=1}^{\infty} \frac{M}{q^i}.
\end{equation*}
Hence, if \eqref{63} holds, then the only expansion of $x$ starting with
$(b_1 \cdots b_n)^-$ is given by $(b_1 \cdots b_n)^-M^{\infty}$.

Finally, if $j > n$, then \eqref{63} holds, for otherwise $b_r=0$ for all $r >n$ again, and $d_j < b_j$ is impossible.
Note that in this case $ q \notin \uu$, because otherwise $(b_{n+i})$ is
the unique expansion of $\sum_{i = 1}^{\infty} \overline{\alpha_i} q^{-i}$
and thus $(d_{n+i}) = (b_{n+i})$ which is impossible due to $j > n$.
Hence, if $q \in \uu$, then $(b_i)$ is the only expansion of $x$ starting with $b_1\cdots b_n$. 
If $q \in \uuu\setminus \uu$ and \eqref{63} holds, then any expansion $(c_i)$ starting with $b_1 \cdots b_n$ is an expansion of $x$ if and only if $(c_{n+i})$ is one of the conjugates of the expansions listed in Theorem  \ref{t32} (vi).
\end{proof}

\section{Proof of Theorems \ref{t13},  \ref{t14} and Corollaries \ref{c16}, \ref{c17}}\label{s7}

Fix $q\in(1,M+1]$ with $(\alpha_i)=\alpha(q)$. 
We recall from Proposition \ref{p27}  that a sequence $(b_i)\in A^{\NN}$ is greedy
if and only if 
\begin{equation}\label{71}
b_{n+1}b_{n+2} \cdots < \alpha_1 \alpha_2 \cdots \qtq{whenever}
n\ge 1
\qtq{and}b_1\cdots b_n\ne M^n.
\end{equation}
We also recall from Proposition \ref{p26} that
\begin{equation}\label{72}
\alpha_{n+1}\alpha_{n+2}\cdots\leq\alpha_1\alpha_2\cdots\qtq{for all}n\ge 0.
\end{equation}

\begin{lemma}\label{l71}
Suppose that $q\in(1,M+1]\setminus\uuu$.  
Then
\begin{enumerate}[\upshape (i)]
\item a greedy sequence $(b_i)$ cannot end with $\overline{(\alpha_i)}$;
\item the set $\uu_q$ is closed;
\item each element $x \in \vv_q \setminus \uu_q$ has a finite greedy expansion.
\end{enumerate}
\end{lemma}

\begin{proof}

(i) Let $q\in(1,M+1]$, and assume that there exists a greedy sequence in base $q$ ending with $\overline{\alpha_1\alpha_2 \cdots}$.
Then \eqref{71} implies that the sequence $\overline{\alpha_1\alpha_2 \cdots}$ itself is also  greedy in base $q$, and therefore
\begin{equation*}
\overline{\alpha_{n+1}\alpha_{n+2} \cdots} < \alpha_1 \alpha_2 \cdots
\qtq{whenever}\alpha_n>0.
\end{equation*}
By Proposition \ref{p31} (ii) this implies that $q\in\uuu$.
\medskip

(ii) Let
$x \in J_q \setminus \uu_q$ and denote the
greedy expansion of $x$ in base $q$ by $(b_i)$.
According to Lemma~\ref{l42}
there exists a positive integer $n$ such that
\begin{equation*}
b_n >0 \quad \mbox{and} \quad \overline{b_{n+1}b_{n+2} \cdots} \geq
\alpha_1 \alpha_2 \cdots .
\end{equation*}
Applying Lemma \ref{l49} and (i) we conclude
that
\begin{equation*}
[x,z] \cap \uu_q = \varnothing
\end{equation*}
for some number $z >x$. It follows that $\uu_q$ is closed from above.
Since the set $\uu_q$ is symmetric it is closed from below as well.
\medskip 

(iii) Assume on the contrary that $a(x,q)=b(x,q)$ for some $x \in \vv_q \setminus \uu_q$.
Then it would follow from Lemmas \ref{l42} and \ref{l47} that for some positive integer $n$,
\begin{equation*}
b_n(x,q) >0 \quad \text{and} \quad \overline{b_{n+1}(x,q) b_{n+2}(x,q) \cdots} = \alpha_1 \alpha_2 \cdots,
\end{equation*}
contradicting (i).
\end{proof}

Recall from Proposition \ref{p31} that the set $\vv$ consists
of those numbers $q >1$
for which the quasi-greedy expansion $(\alpha_i)=\alpha(q)$ satisfies
\begin{equation}\label{74}
\overline{\alpha_{n+1} \alpha_{n+2} \cdots} \leq \alpha_1 \alpha_2 \cdots
\quad \mbox{for all} \quad  n \geq 0.
\end{equation}

If $q \in \vv \setminus \uuu$, then $(\alpha_i)$ is of the form
\begin{equation}\label{75}
(\alpha_i)= (\alpha_1 \cdots \alpha_k
\overline{\alpha_1 \cdots \alpha_k})^{\infty},
\end{equation}
where $k$ is the least positive integer satisfying
\begin{equation}\label{76}
\overline{\alpha_{k+1} \alpha_{k+2} \cdots} = \alpha_1 \alpha_2 \cdots.
\end{equation}
In particular, such a sequence is periodic. Note that $\alpha_k >0$, for otherwise it
would follow from \eqref{74} with $n=k-1$ and \eqref{75} that
\begin{equation*}
M(\alpha_1\cdots \alpha_{k-1}) \le \alpha_1 \cdots \alpha_{k-1}0,
\end{equation*}
which is impossible, because $M>0$ and $\alpha_j \le \alpha_1$ for all $j \ge 1$ by Proposition \ref{p26} (ii). 
Any sequence of the form $(M^m0^m)^{\infty}$, where $m$ is a positive integer, is infinite and satisfies \eqref{72} and \eqref{74} but
not \eqref{52}. On the other hand, there are only countably many
periodic sequences. Hence the set $\vv \setminus \uuu$ is
countably infinite.

Now we are ready to prove Theorems \ref{t13} and \ref{t14}.

\begin{proof}[Proof of Theorem~\ref{t13}]
Throughout the proof $q \in \vv \setminus
\uuu$ is fixed but arbitrary, and $k$ is the least positive integer satisfying \eqref{76} with $(\alpha_i)=\alpha(q)$.

(i) These are the statements of Lemmas \ref{l48} (iii)  and \ref{l71} (ii).
\medskip

(ii) $\abs{\vv_q \setminus \uu_q}=\abs{A_q}=\aleph_0 $, by Lemmas~\ref{l51} and \ref{l71} (iii).
Lemma~\ref{l51} also implies that $A_q$ is dense in $\vv_q$.
It remains to show that all elements of $\vv_q \setminus \uu_q$ are
isolated points of $\vv_q$. 
Since the greedy expansion $b(x,q)$ of a number $x \in \vv_q \setminus \uu_q$  is finite, by Lemma \ref{l49} (i) there exists a number $z>x$ such that $(x,z] \cap \vv_q = \varnothing$.
Since the sets $\uu_q$ and $\vv_q$ are symmetric, there also
exists a number $y < x$ satisfying
\begin{equation*}
[y,x) \cap \vv_q = \varnothing.
\end{equation*}
\medskip 

(iii) and (iv) We already know from Lemma~\ref{l71} that each $x\in\vv_q\setminus\uu_q$ has a finite greedy expansion.
It remains to show that each element
$x \in \vv_q \setminus \uu_q$ has exactly $\aleph_0$ expansions.  
Let $x \in \vv_q \setminus \uu_q$ and let $b_n$ be the last nonzero element of $(b_i)=b(x,q)$. 
If $1 \le j <n$ and $b_j = a_j (x,q) >0$, then 
\begin{equation*}
\overline{a_{j+1} (x,q)\cdots a_n (x,q)}= \overline{(b_{j+1} \cdots b_n)^-} \leq
\alpha_1 \cdots \alpha_{n-j}
\end{equation*}
by Lemma ~\ref{l47}. 
Therefore
\begin{equation}\label{77}
b_{j+1} \cdots b_n > \overline{\alpha_1 \cdots \alpha_{n-j}}.
\end{equation}
Let $(d_i)$ be another expansion of $x$ and let $j$ be the \emph{smallest}
positive integer for which $d_j\not=b_j$. 
Since $(b_i)$ is greedy, we have $d_j<b_j$ and $j\in\set{1,\ldots,n}$. 
First we show that $j \in \set{n-k, n}$.
Assume on the contrary that $j \notin \set{n-k,n}$.

First assume that $n-k < j <n$. 
Then $b_j >0$ and by \eqref{77},
\begin{equation*}
b_{j+1} \cdots b_n 0^{\infty} >
\overline{\alpha_1 \cdots \alpha_{n-j} (\alpha_{n-j+1} \cdots
\alpha_{k})^-} 0^{\infty}.
\end{equation*}
Since $(\alpha_1 \cdots \alpha_{k})^- M^{\infty}$ is the smallest expansion of 1 in base $q$ (see Theorem~\ref{t33} (vi)),
$\overline{(\alpha_1 \cdots \alpha_k)^-}0^{\infty}$ is the greedy expansion of
$M/(q-1) - 1$, and thus
\begin{equation*}
\sum_{i=1}^{\infty} \frac{d_{j+i}}{q^i} =
b_j - d_j + \sum_{i=1}^{\infty} \frac{b_{j+i}}{q^i}  >  \frac{M}{q-1},
\end{equation*}
which is impossible.

Next assume that $1 \leq j < n-k$. Rewriting \eqref{77} one gets
\begin{equation*}
\overline{b_{j+1} \cdots b_n} < \alpha_1 \cdots \alpha_{n-j}.
\end{equation*}
If we had
\begin{equation*}
\overline{b_{j+1} \cdots b_{j+k}} = \alpha_1 \cdots \alpha_k,
\end{equation*}
then
\begin{equation*}
\overline{b_{j+k+1} \cdots b_n} < \alpha_{k+1} \cdots \alpha_{n-j}.
\end{equation*}
Hence
\begin{equation*}
b_{j+k+1} b_{j+k+2} \cdots > \overline{\alpha_{k+1} \alpha_{k+2} \cdots} =  \alpha_1 \alpha_2 \cdots.
\end{equation*}
Since in this case $b_{j+k} = \overline{\alpha_k} < M$, the last
inequality contradicts the fact that $(b_i)$ is a greedy sequence.
Therefore
\begin{equation*}
\overline{b_{j+1} \cdots b_{j+k}} < \alpha_1 \cdots \alpha_k
\end{equation*}
or equivalently
\begin{equation*}
b_{j+1} \cdots b_{j+k} \geq \overline{(\alpha_1 \cdots \alpha_k)^-}.
\end{equation*}
Since $n > j+k$ and $b_n >0$, it follows that
\begin{equation*}
b_{j+1} b_{j+2} \cdots > \overline{(\alpha_1 \cdots \alpha_k)^-}0^{\infty},
\end{equation*}
which leads to the same contradiction as we encountered in the case $n-k<j<n$.
It remains to investigate what happens if $j \in \set{n-k,n}$.

If $j=n-k$, then it follows from \eqref{77} that
\begin{equation*}
b_{n-k+1} \cdots b_n \geq \overline{(\alpha_1 \cdots \alpha_k)^-}.
\end{equation*}
Equivalently,
\begin{equation*}
b_{n-k+1} b_{n-k+2} \cdots =
b_{n-k+1} \cdots b_n 0^{\infty} \geq
\overline{(\alpha_1 \cdots \alpha_k)^-}0^{\infty},
\end{equation*}
and thus
\begin{equation}\label{78}
\sum_{i=1}^{\infty} \frac{d_{n-k+i}}{q^i} \geq
1+ \sum_{i=1}^{\infty} \frac{b_{n-k+i}}{q^i}  \geq \frac{M}{q-1},
\end{equation}
where both inequalities in \eqref{78} are equalities if and only if
\begin{equation*}
d_{n-k}=b_{n-k}^-,\quad 
b_{n-k+1}\cdots b_n=\overline{(\alpha_1 \cdots \alpha_k)^-}\qtq{and}
d_{n-k+1}d_{n-k+2} \cdots = M^{\infty}.
\end{equation*}
Hence $d_{n-k} < b_{n-k}$ is only possible in case $b_{n-k} >0$ and
$b_{n-k+1} \cdots b_n = \overline{(\alpha_1 \cdots \alpha_k)^-}$.

If $j=n$ and $d_n=b_n^-$, then $(d_{n+i})$ is one of the expansions of $1$ in base $q$ listed in Theorem~\ref{t33} (vi).
It follows from Theorem~\ref{t33} (ii)  that $\tilde q \ge (M+2)/2$ and that $\tilde q = (M+2)/2$ if and only if $M$ is even. Since $q \ge \tilde q$, we have $M/(q-1) \le 2$ and $M/(q-1) =2$ if and only if $q = \tilde q$ and $M$ is even. Hence, if $M$ is even, $q = \tilde q$ and $b_n \ge 2$, the number $x$ has one more expansion, namely $(b_1 \cdots b_{n-1})(b_n-2) M^{\infty}$.
\end{proof}

\begin{proof}[Proof of Theorem~\ref{t14}]
Fix $q \in (1,M+1]\setminus \vv$. 
Since $\uu_q \subseteq \overline{\uu_q}  \subseteq\vv_q$ and every $x\in\vv_q\setminus\uu_q$ has a finite greedy expansion by Lemma~\ref{l71} (iii), the required relation $\uu_q=\overline{\uu_q}=\vv_q$ will follow if we show that a number $x \in J_q$ with a finite greedy expansion does \emph{not} belong to $\vv_q$.

Let $x \in J_q$ be an element with a finite greedy expansion. 
Since $q \notin \vv$, 
by Proposition \ref{p31} (iii) there exists a positive integer $n$ such that
\begin{equation*}
\alpha_n > 0
\qtq{and}
\overline{\alpha_{n+1} \alpha_{n+2} \cdots} > \alpha_1 \alpha_2 \cdots .
\end{equation*}
Since $a(x,q)$ ends with $\alpha_1 \alpha_2 \cdots$,
$x \notin \vv_q$ by Lemma~\ref{l47}. 
\end{proof}

We recall that in Sections \ref{s4}--\ref{s7} we have used Definition \ref{d43} of $\vv_q$ and $\vv_q'$ for $q\in(1,M+1]$.
Part (iii) of our next theorem shows that this definition is equivalent to the earlier one given in the introduction:  

\begin{theorem}\label{t72}\mbox{}
For $(x,q) \in \bfj$ we have the following equivalences: 

\begin{enumerate}[\upshape (i)]
\item $x\in\uu_q$ if and only if $x$ has a unique expansion.
\item $x\in\overline{\uu_q}$ if and only if  at least one of $x$ and $M/(q-1)-x$ has a unique infinite expansion.
\item $x\in\vv_q$ if and only if $x$ has at most one doubly infinite expansion.
\end{enumerate}
\end{theorem}

\begin{proof}
(i) This is the definition of $\uu_q$.
\medskip 

(iii) 
If $q=M+1$, then $\vv_q=J_q=[0,1]$ and each number has a unique infinite expansion, hence at most one doubly infinite expansion.
Henceforth we assume that $q\in(1,M+1)$.
Suppose that $x \in \vv_q$. If $x\in\vv_q \setminus \uu_q$ and $q \in \vv$, then by checking the list of expansions of $x$ in  Proposition \ref{p12} (iii) and (iv) and Theorem \ref{t13} (iv) we see that $x$ has precisely  one doubly infinite expansion.
In all other cases we have $x\in\uu_q$ by Theorem \ref{t14}, so that $x$ has a unique expansion.

Conversely, assume that $x\in J_q$ has at most one doubly infinite expansion.
Since $q<M+1$, by Proposition \ref{p22} (ii) the quasi-greedy expansion $(a_i):=a(x,q)$ is doubly infinite, so that it is the unique doubly infinite expansion of $x$.

If $(a_i)$ is also the lazy expansion of $x$, then its reflection is a greedy sequence, so that in particular 
\begin{equation}\label{79}
\overline{a_{n+1}a_{n+2}\cdots}
\le \alpha_1\alpha_2\cdots\qtq{whenever} a_n>0
\end{equation}
by Proposition \ref{p27} (i), where $(\alpha_i)=\alpha(q)$. Hence  $x\in\vv_q$ by Definition \ref{d43}.

Otherwise, the lazy expansion $(c_i)$ of $x$ is not doubly infinite, and, since it is infinite by Lemma \ref{l25}, it has a last digit $c_k<M$. 
Then
\begin{equation*}
(c_1\cdots c_k)^+\overline{\alpha_1\alpha_2\ldots}
\end{equation*}
is also an expansion of $x$. Since this expansion ends with the reflection of $(\alpha_i)$, it is doubly infinite by Proposition \ref{p22} (ii). 
By our hypothesis it coincides with  $(a_i)$:
\begin{equation}\label{710}
(a_i)=(c_1\cdots c_k)^+\overline{\alpha_1\alpha_2\ldots}.
\end{equation}
(Incidentally, \eqref{710} implies by Proposition~ \ref{p27} (ii) and Proposition \ref{p31} (iii) that $q\in\vv$, but we do not need this in the sequel.)
In view of Lemma~\ref{l47} it remains to check the  condition \eqref{79}.
Consider an index $n$ such that $a_n>0$.

If $n<k$, then $c_n=a_n>0$, and therefore $(\overline{c_{n+i}})<(\alpha_i)$ because  $(c_i)$ is lazy. 
Hence $\overline{c_{n+1}\cdots c_k}\le\alpha_1\ldots\alpha_{k-n}$, and using \eqref{710} the condition \eqref{79} follows:
\begin{equation*}
\overline{a_{n+1} \ldots a_{k}}
=\overline{(c_{n+1}\cdots c_k)^+}
<\overline{c_{n+1}\cdots c_k}
\le \alpha_1 \cdots \alpha_{k-n}.
\end{equation*}

The case $n=k$ is obvious because then $(\overline{a_{n+i}})=(\alpha_i)$.
Finally, in case $n>k$ we  argue as follows:
\begin{equation*}
a_n>0
\Longrightarrow \overline{\alpha_{n-k}}>0
\Longrightarrow \alpha_{n-k}<M
\Longrightarrow (\alpha_{n-k+i})\le (\alpha_i)
\Longrightarrow (\overline{a_{n+i}})\le (\alpha_i).
\end{equation*}
\medskip

(ii) For $q= M+1$ the equivalence follows by observing that every $x \in \uuuq=J_q=[0,1]$ has a unique infinite expansion.
Henceforth we assume that $q\in(1,M+1)$.

If $x\in\uu_q$, then $x$ has a unique expansion which must be infinite.
If $x\in\uuuq\setminus\uu_q$, then $q \in \uuu$, and the lists in  Proposition \ref{p12} (iii) and (iv) show again that $x$ has a unique  infinite expansion if and only if $x \in A_q$. By Lemma~\ref{l61} exactly one of the numbers  $x$ and $M/(q-1)-x$ has a unique infinite expansion if $x \in \overline{\uu_q} \setminus \uu_q$.   

Conversely, assume that either $x$ or $M/(q-1)-x$ has a unique infinite expansion in base $q$ (or both). 
Then $x$ also has a unique doubly infinite expansion and therefore $x\in\vv_q$ by the already proved Part (iii) of the present theorem. 
If $q\notin\vv\setminus\uuu$, then we conclude by noting that $\overline{\uu_q}=\vv_q$ in this case. 
If $q\in\vv\setminus\uuu$, then $x\notin\vv_q\setminus\uu_q$.
Indeed, if $q\in\vv\setminus\uuu$, then for every $x \in\vv_q\setminus\uu_q$ both $x$ and $M/(q-1)-x$ have infinitely many infinite expansions by Theorem \ref{t13} (iv), contradicting our assumption. 
Therefore  $x\in\uu_q=\overline{\uu_q}$.
\end{proof}

\begin{proof}[Proof of Corollary \ref{c16}]
If $q\in\uuu$, then $\uu_q$ is not closed by Theorem \ref{t11} (i), (ii).
If $q\in(1,M+1] \setminus\uuu$, then $\uu_q$ is closed by Theorems \ref{t13} (i) and \ref{t14}.
\end{proof}

\begin{proof}[Proof of Corollary \ref{c17}]
The relation 
$q\in\uu\Longleftrightarrow 1\in\uu_q$ is evident and the relation 
$q\in\vv\Longleftrightarrow 1\in\vv_q$ follows from Proposition~\ref{p31} (iii) and Lemma~\ref{l47}. 
It remains to prove the relation $q\in\uuu\Longleftrightarrow 1\in\uuuq$. 
If $q\in\uuu$, then 
\begin{equation*}
\alpha_n>0\Longrightarrow (\overline{\alpha_{n+i}})<(\alpha_i)
\end{equation*}
by Proposition \ref{p31} (ii).  
In particular, we have
\begin{equation*}
\alpha_n>0\Longrightarrow (\overline{\alpha_{n+i}})\le (\alpha_i),
\end{equation*}
so that $1\in\vv_q$ by Lemma~\ref{l47}. 
We conclude by recalling from Theorem \ref{t11} (i) that $\vv_q=\uuuq$.
If $q\in (1,M+1]\setminus\uuu$, then $q \notin \uu$, hence $1 \notin \uu_q$ and thus $1 \notin \overline{\uu_q}$ because $\uu_q$ is closed by Lemma~\ref{l71} (ii). 
\end{proof}

\section{Proof of Theorems~\ref{t18}, \ref{t19}, \ref{t111} and \ref{t112}}\label{s8}

We recall from Lemma~\ref{l44} the following inclusions:
\begin{equation}\label{81}
\uu_p'  \subseteq \uu_q' ,\quad
\vv_p'  \subseteq \vv_q' 
\qtq{and} 
\vv_p'  \subseteq \uu_q'
\qtq{for all} 1< p < q \le M+1.
\end{equation}
In the following lemma we exhibit some cases where these inclusions are not strict. For convenience, we define $\vv_1':=\set{0^{\infty},M^{\infty}}$. 

\begin{lemma}\label{l81}
If $(q_1,q_2)$ is a connected component of $(1, M+1] \setminus \vv$, then 
\begin{align*}
&\uu_q'=\vv_{q_1}'\qtq{for all}q\in(q_1,q_2],
\intertext{and}
&\vv_q'=\uu_{q_2}' \qtq{for all}q\in [q_1,q_2).
\end{align*}
\end{lemma}

\begin{proof} 
The case $(q_1,q_2)=(1,\tilde q)$ follows from Examples \ref{e45} (i).
Henceforth we assume that $q_1\in\vv\setminus\uu$; see Theorem \ref{t33} (iii). 
Let us write 
\begin{equation*}
\alpha(q_2)=(\alpha_1 \cdots \alpha_k
\overline{\alpha_1 \cdots \alpha_k})^{\infty}
\end{equation*}
where $k$ is chosen to be minimal; then $\alpha_k>0$ by Theorem \ref{t33} (v).
Due to \eqref{81}, it is sufficient to show that $\uu_{q_2}'  \subseteq \vv_{q_1}'$. 
Suppose that a sequence $(c_i) \in A^{\NN}$
is univoque in base $q_2$, i.e.,
\begin{equation}\label{82}
c_{n+1} c_{n+2} \cdots < (\alpha_1 \cdots \alpha_k
 \overline{\alpha_1 \cdots \alpha_k})^{\infty} \quad \mbox{whenever} \quad
c_n < M
\end{equation}
and
\begin{equation}\label{83}
\overline{c_{n+1} c_{n+2} \cdots} < (\alpha_1 \cdots \alpha_k
 \overline{\alpha_1 \cdots \alpha_k})^{\infty} \quad \mbox{whenever} \quad   c_n >0.
\end{equation}
If $c_n < M$, then by \eqref{82},
\begin{equation*}
c_{n+1} \cdots c_{n+k} \leq \alpha_1 \cdots \alpha_k.
\end{equation*}
If we had
\begin{equation*}
c_{n+1} \cdots c_{n+k} =  \alpha_1 \cdots \alpha_k,
\end{equation*}
then
\begin{equation*}
c_{n+k+1} c_{n+k+2} \cdots < (\overline{\alpha_1 \cdots \alpha_k}
 \alpha_1 \cdots \alpha_k)^{\infty},
\end{equation*}
and by \eqref{83} (note that in this case $c_{n+k}=\alpha_k >0$),
\begin{equation*}
c_{n+k+1} c_{n+k+2} \cdots >(\overline{\alpha_1 \cdots \alpha_k}
 \alpha_1 \cdots \alpha_k)^{\infty},
\end{equation*}
a contradiction. Hence
\begin{equation*}
c_{n+1} \cdots c_{n+k} \leq (\alpha_1 \cdots \alpha_k)^-.
\end{equation*}
Note that $c_{n+k} < M$ in case of equality.
It follows by induction that
\begin{equation*}
c_{n+1} c_{n+2} \cdots \leq ((\alpha_1 \cdots \alpha_k)^-)^{\infty}.
\end{equation*}
Since a sequence $(c_i)$ satisfying \eqref{82} and \eqref{83} is
infinite, we conclude from
Proposition~\ref{p27} (ii) and Theorem~\ref{t33} (iv) and (v) that  $(c_i)$ is the
quasi-greedy expansion of some $x$ in base $q_1$. Repeating the
above argument for the sequence $\overline{c_1 c_2 \cdots}$, which
is also univoque in base $q_2$, we conclude from Lemma \ref{l47} that $(c_i) \in
\vv_{q_1}'$.  Hence $\uu_{q_2}'  \subseteq  \vv_{q_1}'$. 
\end{proof}

\begin{proof}[Proof of Theorem \ref{t18}]
It follows from Theorems \ref{t11} and \ref{t13} and the inclusions \eqref{81} that 
\begin{equation*}
\vv_s'  \subseteq \uu_q'   \subsetneq \vv_q'  \subseteq \uu_r'
\qtq{if} q \in \vv \qtq{and} 1< s< q < r \le M+1.
\end{equation*}
Hence the stability intervals $(q_1,q_2]$ of Lemma \ref{l81} for $\uu_q$ and the stability intervals $(1, \tilde q)$ and $[q_1,q_2)$ with $q_1 \in \vv \setminus \uu$ for $\vv_q$, are maximal.
By Theorem~\ref{t33} (iii), these stability intervals for $\uu_q$ and $\vv_q$ cover $(1,M+1] \setminus \uuu$ and $(1,M+1] \setminus \uu$, respectively. 
We conclude the proof by recalling that $\uuu$ has no interior points, and therefore $\uu$ and $\uuu$ do not contain any non-degenerate interval.
\end{proof}

\begin{proof}[Proof of Theorem \ref{t19}]
(i) This is immediate from Theorem~\ref{t11} (i) and Proposition~\ref{p12} (i).
\medskip 

(ii) Since $\vv_q$ is a closed set that contains the endpoints of $J_q$, the components of $J_q\setminus\vv_q$ are open intervals $(x_L,x_R)$ indeed, and their endpoints belong to $\vv_q$.
By Lemmas \ref{l48} (ii) and  \ref{l51} the elements of $\uu_q$ cannot be endpoints, hence the endpoints belong to $\vv_q\setminus\uu_q=A_q\cup B_q$.
Note that $\abs{A_q}=\abs{B_q}=\aleph_0$ by Theorem \ref{t11} (ii), and Proposition \ref{p12} (ii). 

If $x\in A_q$, then $x$ is a right isolated point of $\vv_q$ by Lemma \ref{l49} (i), and a left accumulation point of $\vv_q$ by Lemma \ref{l52}, so that $x$ is a left endpoint $x_L$ but not a right endpoint $x_R$. 
Applying Lemma \ref{l61}, we infer that every $x\in B_q$ is a right endpoint $x_R$ but not a left endpoint $x_L$.

It remains to show that if $b(x_L,q)=b_1 \cdots b_n 0^{\infty}$ with $b_n>0$, then $b(x_R,q)=b_1 \cdots b_n \overline{\alpha(q)}$. 
First we show that $(d_i):=b_1\cdots b_n\overline{\alpha(q)}$ is a greedy sequence that belongs to $\vv_q'\setminus\uu_q'$. 
Since the sequence ends with $b_n\overline{\alpha(q)}$ and $b_n>0$, it does not belong to $\uu_q'$ by Lemma \ref{l42}.
Since $(d_i)$ is infinite, it remains to verify that $d_k< M \Longrightarrow (d_{k+i})< \alpha(q)$ and $d_k>0 \Longrightarrow (\overline{d_{k+i}}) \le \alpha(q)$. These implications hold true if $k \ge n$ because $q \in \uuu$. Suppose that $1 \le k < n$.  
Assume first that $d_k=b_k<M$.
Since $b(x_L,q)=b_1\cdots b_n0^{\infty}$ is greedy, we have $b_{k+1} \cdots b_n0^{\infty}<\alpha(q)$, and hence $b_{k+1} \cdots b_n \le \alpha_1\cdots\alpha_{n-k}$.
It remains to observe that $\overline{\alpha(q)} < (\alpha_{n-k+i})$ by Proposition \ref{p31} (ii).
Now assume that $d_k=b_k>0$.
Since $x_L\in\vv_q$ and $a(x_L,q)=(b_1\cdots b_n)^-\alpha(q)$, we have $\overline{(b_{k+1}\cdots b_n)^-\alpha(q)}\le\alpha(q)$ by Lemma \ref{l47}.
This implies in particular the required relation $\overline{b_{k+1}\cdots b_n}\alpha(q)\le\alpha(q)$. 

Since $b_1\cdots b_n\overline{\alpha(q)}\in\vv_q'\setminus\uu_q'$ and $a_1(x_R,q)\cdots a_n(x_R,q)\ge b_1 \cdots b_n$, it remains to observe that if $(c_i)\in\vv_q'$ starts with $b_1\cdots b_n$, then  $(c_i)\ge b_1\cdots b_n\overline{\alpha(q)}$, which follows from \eqref{46} and the fact that $b_n>0$.
\medskip

(iii) Since $q\notin\uuu$, $\uu_q$ is closed and contains the endpoints of $J_q$, so that $J_q\setminus\uu_q$ is the union of  disjoint open intervals $(x_L,x_R)$.
Since $q\in\vv\setminus\uuu$, $\vv_q\setminus\uu_q$ is a discrete set by Theorem \ref{t13} (ii). 
The relation \eqref{12} follows from Lemmas \ref{l48} and \ref{51}. The relation \eqref{subsequent} is the same as the relation between $b(x_L,q)$ and $b(x_R,q)$ in (ii) and can also be proved along the exact same lines, except that we now have to invoke Proposition~\ref{p31} (iii) in place of Proposition~\ref{p31} (ii).  
\medskip 

(iv) This follows from Examples \ref{e45} (i).
\medskip 

(v) We already know from Lemma~\ref{l81} that $\uu_q' = \vv_{q_1}'$. Write the set $J_q \setminus \uu_q = \cup^*(x_L,x_R)$ as a disjoint union of open intervals $(x_L, x_R)$, and define the map $h : J_q \to J_{q_1}$ as follows: 
\begin{equation*}
\begin{cases}
h(x)=\sum_{i=1}^{\infty} c_i q_1^{-i} &\text{ if } x= \sum_{i=1}^{\infty} c_i q^{-i} \in \uu_q,\\
h(x)=\frac{x-x_L}{x_R-x_L} (h(x_R)-h(x_L)) + h(x_L) &\text{ if }  x\in (x_L,x_R). \\
\end{cases}
\end{equation*}
The map $h$ is strictly increasing by Proposition ~\ref{p27}. It is clear that $h$ restricted to each closed interval $[x_L,x_R]$ is continuous. It remains to observe that $h$ cannot have a jump discontinuity at an accumulation point of $\uu_q$ by Lemmas ~\ref{l29} and ~\ref{l210}. Hence $h$ is a strictly increasing bijection (and therefore a homeomorphism) that maps  $J_q \setminus \uu_q$ onto $J_{q_1} \setminus \vv_{q_1}$. 
\end{proof}

\begin{lemma}\label{l82}
Let $q \in (1,M+1]$. The set $\vv_q'$ is a compact subset of $A^{\NN}$ if and only if $q \not= M+1$.
\end{lemma}
\begin{proof} If $q \not=M+1$, then a sequence $(c_i) \in A^{\NN}$ belongs to $\vv_q'$ if and only if \eqref{43} and \eqref{44} hold. Hence $A^{\NN} \setminus \vv_q'$ is open, whence $\vv_q'$ is closed and thus compact. For $n \ge 1$, the sequence $10^n1^{\infty}$ belongs to $\vv_{M+1}'$. If $ n \to \infty$, then $10^n1^{\infty}$ converges to $10^{\infty}$ which does not belong to $\vv_{M+1}'$, i.e., $\vv_{M+1}'$ is not closed. 
\end{proof} 
  
\begin{proof}[Proof of Theorem~\ref{t111}.]
(ia) This is the classical integer base case.
\medskip

(ib) 
If $\vv_q$ had an interior point, then bij Lemma~\ref{l29}, $\vv_q$ would also have an interior point with a finite greedy expansion, contradicting Lemma~\ref{l49}.  By Theorem \ref{t11} (ii), $\vv_q\setminus\uu_q$ is dense in $\vv_q$. Hence every $x\in\uu_q$ is an accumulation point of $\vv_q$. Since the accumulation points of a set form a closed set, we infer that every $x\in\uuuq=\vv_q$ is an accumulation point of $\vv_q$.
\medskip

(iia) See Examples \ref{e45} (i); here we have $q_1=\tilde q$.
\medskip

(iib) Theorem \ref{t13} (ii), Lemma~\ref{l81} and induction on $n$ show that $\uu_q$ is countably infinite for each $q \in (\tilde q, q_{KL})$. Suppose that $q \in (q_n,q_{n+1}]$. According to Lemma~\ref{l51}, for each element $x \in \uu_{q_n}$, there is a sequence $(x_i)$ of elements in $\vv_{q_n} \setminus \uu_{q_n}$ such that $a(x_i,q_{n}) \to a(x,q_{n})$ as $i \to \infty$. Since $\uu_q' = \vv_{q_n}'$, all elements of $\pi_q (\uu_{q_n}')$ are accumulation points of $\uu_q$ and can be approximated arbitrarily closely by elements of $\pi_q (\vv_{q_n}' \setminus \uu_{q_n}')$. A number $x \in \vv_{q_n} \setminus \uu_{q_n}$ is isolated in $\vv_{q_n}$ by Theorem~\ref{t13} (ii). Lemma~\ref{l21} implies that $\pi_q(a(x,q_n))$ is isolated in $\uu_q$ because univoque sequences are in particular greedy and quasi-greedy. 
\medskip

(iiia)  The set $\vv_{q_0}$ is perfect by Theorem~\ref{t11} and hence consists entirely of condensation points. Suppose that a sequence $(c_i)$ is univoque in base $q$, and let $W$ be an arbitrary neighborhood of $\sum_{i=1}^{\infty} c_i q^{-i}$. If $N$ is large enough, then each sequence starting with $c_1 \cdots c_N$ is the expansion in base $q$ of a number in $W$. If $(c_i)$ also belongs to $\uu_{q_0}'$, then, since univoque sequences are in particular greedy and quasi-greedy, applying Lemma~\ref{l21} and using the fact that $\vv_{q_0}' \setminus \uu_{q_0}'$ is countable (see Theorem \ref{t11}), we conclude that there are uncountably many sequences in $\uu_{q_0}'$ starting with $c_1 \cdots c_N$. Since, moreover, $\uu_{q_0}'  \subseteq \uu_q'$, each number belonging to $\pi_q (\uu_{q_0}')$ is a condensation point of $\uu_q$. It follows easily from Lemmas ~\ref{l21}, ~\ref{l52}, ~\ref{l61} and ~\ref{l82} that $\overline{\uu_{q_0}'}=\vv_{q_0}'$ because $q_0 \in \uuu \setminus \set{M+1}$.  Since $\pi_q$ is continuous and since the condensation points of $\uu_q$ form a closed set, each element in $\uu_q=\pi_q(\vv_{q_0}')$ is a condensation point of $\uu_q$. The set $\uu_q$ has no interior points because $\vv_q$  has none; see the proof of (ib). Hence $\uu_q$ is a Cantor set. It follows from Lemma~\ref{l51} that $\pi_q(\vv_{q_0}' \setminus \uu_{q_0}')$ is dense in $\uu_q$. 
\medskip

(iiib) One shows exactly as in (iib) that elements of $\pi_q(\vv_{q_n}' \setminus \uu_{q_n}')$ are isolated points of $\uu_q$ and form a dense subset of $\uu_q$. Hence, elements of $\pi_q(\uu_{q_n}')$ are accumulation points of $\uu_q=\pi_q(\vv_{q_n}')$. Since $\pi_q(\vv_{q_0}')$ is compact by Lemma~\ref{l82}, and since $\uu_{q_n}' \setminus \vv_{q_0}'$ is countable, elements of  $\pi_q(\uu_{q_n}'\setminus \vv_{q_0}')$ are not condensation points of $\uu_q$. Numbers belonging to $\pi_q(\vv_{q_0}')$ are condensation points of $\uu_q$ as follows from the reasoning in (iiia). 
\end{proof}

\begin{proof}[Proof of Theorem~\ref{t112}.]
Since we have always (ii) $\Longrightarrow$ (iii) $\Longrightarrow$ (iv), it suffices to show that $\uu'_q$ is a shift of finite type for every $q \in (1, M+1] \setminus \uuu$, and that $\uu'_q$ is not closed if $q\in\uuu$. Fix $q \in (1, M+1] \setminus \uuu$. Consider the connected component $(q_1, q_2)$ of $(1, \infty) \setminus \vv$, satisfying $q \in (q_1, q_2]$.
Let us write
\begin{equation*}
\alpha(q_2)=(\alpha_i) =
(\alpha_1 \cdots \alpha_k \overline{\alpha_1 \cdots \alpha_k})^{\infty},
\end{equation*}
where $k$ is minimal, and set
\begin{equation*}
\mathcal{F}=\set{ja_1\cdots a_k\in  A^{k+1}\ :\ j< M \qtq{and}
a_1 \cdots a_k \geq \alpha_1 \cdots \alpha_k}.
\end{equation*}
It suffices to show that a sequence $(c_i)\in A^{\NN}$ belongs to $\uu_q'=\uu'_{q_2}$ if and only if
\begin{equation}\label{86}
c_j \cdots c_{j+k} \notin \mathcal{F}
\qtq{and}
\overline{c_j \cdots c_{j+k}} \notin \mathcal{F}
\qtq{for all}
j \geq 1.
\end{equation}
It follows from the proof of Lemma ~\ref{l81} that $(c_i) \in \uu_{q_2}'$ if and only if 
\begin{equation*}
c_j<M \Longrightarrow c_{j+1}\cdots c_{j+k}\le(\alpha_1\cdots\alpha_k)^- \text{ and }  
c_j >0 \Longrightarrow \overline{c_{j+1}\cdots c_{j+k}}\le(\alpha_1\cdots\alpha_k)^-, 
\end{equation*}
and this is equivalent to \eqref{86}.  

Finally, if $q \in \uuu$, then $\uu_q'$ is not closed as follows from Lemma~\ref{l52}. 
\end{proof}

\begin{proof}[Proof of Theorem \ref{t113}]
First we show that $\bfv  \subseteq\bfuu \cap \bfj$. Fix $(x,q)\in\bfv$. If $q=M+1$, then $(x,q) \in \bfuu$ because $\overline{\uu_q}=\vv_q$.
If $1< q < M+1$, then by Lemma~\ref{l44}, $a(x,q)\in\uu'_r$ for every $r \in (q,M+1]$, so that $\pi_r(a(x,q))\in\uu_r$.
Since  $\pi_r(a(x,q))\to \pi_q(a(x,q))=x$ as $r\downarrow q$, we conclude that $(x,q)\in\bfuu \cap \bfj$.

Since $\bfu \subseteq\bfv$, the converse inclusion $\bfuu \cap \bfj  \subseteq\bfv$ will follow if we show that $\bfv$ is (relatively) closed in $\bfj$, i.e., 
if $(x,q)\in\bfj \setminus\bfv $,  then $(x',q')\notin\bfv $ for  all $(x',q')\in\bfj $ close enough to $(x,q)$.

Henceforth we assume that $q\in(1,M+1)$, and write $(\beta_i)=\beta(q)$, $(\beta_i')=\beta(q')$, $(a_i)=a(x,q)$ and $(a_i')=a(x',q')$.
By Lemma \ref{l47} there exist two positive integers $n$ and $m$ such that
\begin{equation}\label{91}
a_{n}>0\qtq{and}\overline{a_{n+1}\ldots a_{n+m}}>\beta_1\ldots\beta_m .
\end{equation}
It follows from the definition of quasi-greedy expansions that
\begin{equation*}
\frac{a_1}{q}+\cdots+\frac{a_{j-1}}{q^{j-1}}+ \frac{a_j^+}{q^j}+\frac{1}{q^{j+m}}>x\qtq{whenever}a_j<M.
\end{equation*}
Hence, if $(x',q') \in \bfj$ is sufficiently close to $(x,q)$, then, applying also Lemma \ref{l21} (i), (iii), 
\begin{gather}
\frac{a_1}{q'}+\cdots+\frac{a_{j-1}}{(q')^{j-1}}+ \frac{a_j^+}{(q')^j}+\frac{1}{(q')^{j+m}}>x'
\text{ whenever }
j\le n+m\text{ and }a_j<M,\label{92}\\
\beta'_1\ldots \beta'_m\le \beta_1\ldots \beta_m
\qtq{and}
a_1' \ldots a_{n+m}' \ge a_1 \ldots a_{n+m}.\label{93}
\end{gather}
Now we distinguish between two cases.

If $a_1'\ldots a'_{n+m} =  a_1\ldots a_{n+m}$, then we have
\begin{equation*}
a'_n>0\quad\text{and}\quad \overline{a'_{n+1}\ldots a'_{n+m}}>\beta_1\ldots\beta_m\ge  \beta'_1\ldots \beta'_m
\end{equation*}
by \eqref{91} and \eqref{93}. 
This proves that $(x',q')\notin\bfv $.

If $a_1'\ldots a'_{n+m}> a_1\ldots a_{n+m}$, then let us consider the smallest $j$ for which $a'_j>a_j$. 
It follows from \eqref{91}, \eqref{92} and \eqref{93} that
\begin{equation*}
a'_j=a_j^+>0\quad\text{and}\quad \overline{a'_{j+1}\ldots a'_{j+m}}
=M^m>\beta_1\ldots\beta_m\ge \beta'_1\ldots\beta'_m.
\end{equation*}
Hence $(x',q')\notin\bfv $ again.
\end{proof}

\section{List of principal terminology and notations}\label{s9}

\begin{itemize}
\item Page 1
\begin{itemize}
\item $\NN:=\set{1,2,3,\ldots}$
\item \emph{alphabet} $A:=\set{0,1,\ldots,M}$
\item \emph{digit}: an element of the alphabet
\end{itemize}
\item Page 2
\begin{itemize}
\item \emph{sequence}: an element of $A^{\NN}$
\item \emph{block} or \emph{word}: a finite sequence of digits
\item \emph{conjugate} or \emph{reflection} of a digit, word, or a sequence: 
\begin{equation*}
\overline{c_i}:=M-c_i,\quad
\overline{c_1\cdots c_n}:=\overline{c_1}\mbox{ } \cdots\mbox{ }\overline{c_n},\quad
\overline{c_1c_2\cdots}:=\overline{c_1}
\mbox{ }\overline{c_2}\mbox{ }\cdots
\end{equation*}
\item $w^+:=c_1\cdots c_{n-1}(c_n+1)$ if $w=c_1\cdots c_{n-1}c_n$ and $c_n<M$
\item $w^-:=c_1\cdots c_{n-1}(c_n-1)$ if $w=c_1\cdots c_{n-1}c_n$ and $c_n>0$
\item \emph{lexicographical order between words and sequences}
\item \emph{finite, co-finite, infinite, co-infinite and doubly infinite sequences}
\item \emph{expansion (of a number x in base $q$ over the alphabet $A$)}:
\begin{equation*}
(c_i)\in A^{\NN} \text{ satisfying } x=\sum_{i=1}^{\infty}\frac{c_i}{q^i}
\end{equation*}
\item short notation:
\begin{equation*}
\pi_q(c)=\pi_q(c_1c_2\cdots):=\sum_{i=1}^{\infty}\frac{c_i}{q^i},\quad c=(c_i)\in A^{\NN}
\end{equation*}
We assume  after Page 2 and in the remainder of this list that $1 < q  \le M+1$. 
\item 
$J_q:=[0,M/(q-1)]$: the set of numbers having an expansion in base $q$
\item $b(x,q)$: the \emph{greedy} (or lexicographically largest) expansion of $x$ in base $q$
\item $a(x,q)$: the \emph{quasi-greedy} (or lexicographically largest infinite)  expansion of $x$ in base $q$
\end{itemize}
\item Page 3
\begin{itemize}
\item $\uu$ is the set of \emph{univoque bases}. A base $q$ belongs to $\uu$ if $x=1$ has a unique expansion in base $q$. 
\item $\vv$ is the set of bases $q$ in which $x=1$ has a unique doubly infinite expansion.
\item A \emph{Cantor set} is a nonempty closed set having neither interior nor isolated points. 
\item $\uuu$ is the topological closure of $\uu$.
\end{itemize}
\item Page 4
\begin{itemize}
\item $\uu_q$ is the set of numbers $x\in J_q$ having a unique expansion in base $q$.
\item $\vv_q$ is the set of numbers $x\in J_q$ having at most one doubly infinite expansion in base $q$.
\item $\overline{\uu_q}$ is the topological closure of $\uu_q$.
\item $A_q$ is the set of numbers $x \in \vv_q \setminus \uu_q$ such that $b(x,q)$ is finite. 
\item $B_q$ is the set of numbers $x \in \vv_q \setminus \uu_q$ such that $b(x,q)$ is infinite. 
\item $\ell: J_q \to J_q$ is the \emph{reflection map} defined by $\ell(x)=M/(q-1)-x$. 
\end{itemize} 
\item Page 5
\begin{itemize}
\item The number $\tilde q:=\min\vv$ is the smallest element of $\vv$; for $M=1$, $\tilde q$ is the Golden ratio; see also Theorem \ref{t33}.
\end{itemize}
\item Page 6
\begin{itemize}
\item $\uu_q'\subseteq A^{\NN}$: the set of expansions of the elements of $\uu_q$
\item $\vv_q'\subseteq A^{\NN}$: the set of quasi-greedy expansions of the elements of $\vv_q$
\end{itemize}
\item Page 7
\begin{itemize}
\item The \emph{Komornik--Loreti constant} $q_{KL}:=\min\uu$ is the smallest element of $\uu$; see also Theorem \ref{t32}.
For $M=1$, $q_{KL}\approx 1.787$.
\end{itemize}
\item Page 8
\begin{itemize}
\item The two-dimensional analogues of $\uu_q$, $\vv_q$ and $J_q$ are defined as follows. 
\begin{align*}
\bfu:&=\set{(x,q)\in\RR^2\ :\ q\in(1,M+1]\text{ and }x\in\uu_q},\\
\bfv:&=\set{(x,q)\in\RR^2\ :\ q\in(1,M+1]\text{ and }x\in\vv_q},\\
\bfj:&=\set{(x,q)\in\RR^2\ :\ q\in(1,M+1]\text{ and }x\in J_q}. 
\end{align*}
\end{itemize}
\item Page 11
\begin{itemize}
\item $\beta(q):=b(1,q)$: greedy expansion of $x=1$ in base $q$
\item $\alpha(q):=a(1,q)$: quasi-greedy expansion of $x=1$ in base $q$
\end{itemize}
\item Page 12
\begin{itemize}
\item $X\subseteq\RR$ is \emph{closed from above (below)} if the limit of every bounded decreasing (increasing) sequence in $X$ belongs to $X$.
\end{itemize}
\end{itemize}

\noindent 
\emph{Acknowledgement.} We thank the referee for his/her helpful remarks.

\end{document}